\documentclass[pdflatex,sn-mathphys-num, 12pt]{sn-jnl}

\usepackage{graphicx}%
\usepackage{multirow}%
\usepackage{amsmath,amssymb,amsfonts}%
\usepackage{amsthm}%
\usepackage{mathrsfs}%
\usepackage[title]{appendix}%
\usepackage{xcolor}%
\usepackage{textcomp}%
\usepackage{manyfoot}%
\usepackage{booktabs}%
\usepackage{algorithm}%
\usepackage{algorithmicx}%
\usepackage{algpseudocode}%
\usepackage{listings}%
\usepackage{bm}
\usepackage{bbm}

\theoremstyle{thmstyleone}%
\newtheorem{theorem}{Theorem}

\newtheorem{lemma}{Lemma}%

\newtheorem{proposition}{Proposition}%
\newtheorem{remark}{Remark}%

\theoremstyle{thmstylethree}%

\newcommand{\bs}{\boldsymbol}

\begin{document}

\title[Article Title]{\bf On the Lebesgue constant of the Morrow-Patterson points}

\author[1]{\fnm{Tomasz} \sur{Beberok}}\email{tomaszbeberok@gmail.com}
\equalcont{These authors contributed equally to this work.}

\author[2]{\fnm{Leokadia} \sur{Bialas-Ciez}}\email{leokadia.bialas-ciez@uj.edu.pl}
\equalcont{These authors contributed equally to this work.}

\author*[3]{\fnm{Stefano} \sur{De Marchi}} \email{stefano.demarchi@unipd.it}
\equalcont{These authors contributed equally to this work.}

\affil[1]{\orgdiv{Department of Applied Mathematics}, \orgname{University of Agriculture in Krakow}, \orgaddress{\country{Poland}}}

\affil[2]{\orgdiv{Institute of Mathematics}, \orgname{Jagiellonian University, Krak\'ow}, \orgaddress{\country{Poland}}}

\affil*[3]{\orgdiv{Department of Mathematics "Tullio Levi-Civita"}, \orgname{University of Padova, Padova}, \orgaddress{\country{Italy}}}


\abstract{
The study of interpolation nodes and their associated Lebesgue
constants are central to numerical analysis, impacting the stability and
accuracy of polynomial approximations. In this paper, we will explore the
Morrow-Patterson points, a set of interpolation nodes introduced to
construct cubature formulas of a minimum number of points in the square
for a fixed degree $n$. We prove that their Lebesgue constant growth is ${\cal O}(n^2)$ as was conjectured based on numerical evidence about twenty years ago in the paper by Caliari, M., De Marchi, S., Vianello, M., {\it Bivariate polynomial interpolation on the square at new nodal sets}, Appl. Math. Comput. 165(2) (2005), 261–274.
}




\maketitle


\section{Introduction}
We start referring to the original work by Morrow, Patterson \cite{MP}, in which for the first time the following set of interpolation points in the square $Q:=[-1,1]^2$ was studied. 

\vskip 1mm

Consider the set of {\it Morrow-Patterson points} for a positive integer $n$
\begin{equation*}MP_n := \{ (x_m, y_k)\in Q\: : \: m=1,...,n+1,\: k=1,...,\tfrac{n}{2}+1\} 
\end{equation*}
where
\begin{equation*} \label{MP_points}
x_m := \cos \frac{m\pi}{n+2}, \qquad y_k :=
\begin{cases}
  \displaystyle \cos  \frac{2 k \pi}{n+3},                      & \text {if $m$ odd,} \\ \\
  \displaystyle \cos \frac{(2 k-1) \pi}{n+3}    & \text{if $m$ even.}
  \end{cases}
\end{equation*}
This set was introduced to construct cubature formulas of minimum number of points in the square $Q$ for a fixed degree $n$, see \cite{MP}.

Denote by $\mathbb{P}_n^2$ the space of the bivariate real polynomials of total degree not greater than $n$. The dimension of $\mathbb{P}_n^2$ is equal to $\binom{n+2}{n}$ and this is also the number of points in the set $MP_n$. Moreover, this set is {\it unisolvent for polynomials} in $\mathbb{P}_n^2$, i.e. there exists a unique polynomial from $\mathbb{P}_n^2$ interpolating $MP_n$, see \cite{PD1}, \cite{PD2}, \cite{B01}. The interpolation at $MP_n$ is given by the projection $\mathcal{C}(Q) \longrightarrow \mathbb{P}_n^2$ for the discrete inner product corresponding to the measure $\sqrt{1-x^2}\sqrt{1-y^2}dxdy$ (see Theorem \ref{cub_formula} below). Moreover, the cubature formula of order $n$ at $MP_n$ for the above measure holds for all bivariate polynomials of degree $2n$. This is not the case for the Padua points of degree $n$ denoted by $Pad_n$,  because the cubature formula for $Pad_n$ does not hold for $T_{2n}$, with $T_{2n}$ the univariate Chebyshev polynomial of the first kind of degree $2n$. Consequently, the projection with a discrete inner product is not an interpolation operator.  

As for the growth of the Lebesgue constant, Bos \cite{B01} established that $\Lambda_n^{MP} \le  {\cal O}(n^6)$ (for the definition of $\Lambda_n^{MP}$ see (\ref{leb_constant})). Later, interpreting interpolation at the Morrow-Patterson points as hyperinterpolation, it was shown in \cite{DeMSV14} that $\Lambda_n^{MP}\leq {\cal O}(n^3)$, and it was conjectured that the actual order of growth is $\Lambda_n^{MP} \le {\cal O}(n^2)$. The main goal of this paper is to provide the proof that $\Lambda_n^{MP} \le {\cal O}(n^2)$.

The paper consists of seven sections. Section 2 is devoted to three independent ways to construct the Morrow-Patterson points. The third discusses more properties of these points, the interpolant at $MP_n$, the Lebesgue function, the Lebesgue constant, and recall the known growth estimates. In Subsection \ref{Sec_cubature} we also present the cubature formula at the $MP_n$. Section 4, explores the use of $MP_n$ points in creating optimal polynomial meshes. Section 5 focuses on trigonometric identities and summation formulas needed to prove the main theorem. The main section, Section 6, proves that ${O}(n^2)$ is the growth rate of the Lebesgue constant for the $MP_n$. In the last section, we present some numerical tests and we shortly describe the Matlab script, {\tt Leb\_MP.m} that readers can use to reproduce the numerical experiments.

\vskip 2mm
\section{Three constructions of $MP_n$ points}
In the present paper, $n$ is usually a positive $\emph{even}$ integer and  $C_n:=\frac 8{(n+2)(n+3)}$.
\vskip 3mm
We now recall three independent ways to construct or to see the Morrow-Patterson points.
\subsection{Construction related to Lissajous curve}

Following \cite{erb, erb3}, 
for $\bm{q}=(q_1,q_2) \in \mathbb{R}^2$, $\bm{\alpha}=(\alpha_1,\alpha_2) \in \mathbb{R}^2$ and $\bm{u}=(u_1,u_2) \in \{-1,1\}^2$, we define the \emph{Lissajous curves} $l^{(\bm{q})}_{\bm{\alpha},\bf{u}}$ by
\begin{equation*}
l^{(\bm{q})}_{\bm{\alpha},\bf{u}} :\mathbb{R}\rightarrow [-1,1]^2, \quad l^{(\bm{q})}_{\bm{\alpha},\bf{u}}(t) = (u_1 \cos(q_1t-\alpha_1), (u_2 \cos(q_2t-\alpha_2) ).
\end{equation*}
These curves can be of two types depending if they have corner points or not. Precisely, they are
\begin{itemize}
\item[-] \emph{degenerate}, if there exists $t' \in \mathbb{R}$ and  $\bm{u'} \in \{-1,1\}^2$, such that $l^{(\bm{q})}_{\bm{\alpha},\bm{u}}(\cdot -t') = l^{(\bm{q})}_{\bm{0},\bm{u}'}$,
\item[-] \emph{non-degenerate}, otherwise.
\end{itemize}
By scaling the parameter $t$, we can see that the minimal period of the Lissajous curves is $2\pi$.

Here, we confine ourselves to the degenerate Lissajous curves of the type
\begin{equation*}
 l^{(n,n+p)}_{\bm{0},\bm{1}}(t)= \biggl(\cos(nt),\cos((n+p)t)\biggl),
\end{equation*}
where $n$ and $p$ are positive integers, such that $n$ and $n+p$ are relatively prime.
To simply the notation, we write 
$$\gamma_{n,p} := l^{(n,n+p)}_{\bm{0},\bm{1}}\,.$$

\begin{itemize}
 \item We notice that, in this case,  we can restrict the parametrization of the curve to the interval $[0, \pi]$. The points $\gamma_{n,p}(0) = (1,1)$ and $\gamma_{n,p}(\pi) = ((-1)^n,(-1)^{n+p})$ denote the starting and the end point of the curve respectively. 

\item We also remark that the curve $\gamma_{n,p}$ is an algebraic curve given by $T_{n+p}(x) - T_n(y)=0 $, where $T_n(x)=\cos(n\arccos(x))$ denotes the classical Chebyshev polynomial of degree $n$ of the first kind.
\item In the case $p=1$ we get the well-known Lissajous curve associated to the {\it Padua points} (cf. \cite{PD1, erb}).
\end{itemize}

The Morrow-Patterson points are the self-intersection points in the interior of the square $[-1,1]^2$ of the Lissajous curve
\begin{equation}\label{MP_curve}
\gamma_{n,1}(t) = \biggl(-\cos((n+3)t),-\cos((n+2)t)\biggl).
\end{equation}
If we sample the curve $\gamma_{n,1}$ along the $(n+2)(n+3)+1$ equidistant points of the interval $[0,\pi]$:
\begin{equation*}
t_k := \frac{\pi k}{(n+2)(n+3)}, \quad k = 0, \dots, (n+2)(n+3),
\end{equation*}
we get the set of node points
\begin{equation*}
Pad_{n+2} := \{ \gamma_{n,1}(t_k) :  k = 0, \dots,(n+2)(n+3)\}
\end{equation*}
which consists of the Padua points of degree $n+2$. After subtracting the points along the edges of the square $[-1,1]^2$ (including the two vertices touched by the Padua points), we get the set of Morrow-Patterson points, i.e.
\begin{equation*}
MP_n = Pad_{n+2}-\text{edges points}. 
\end{equation*}
It should be noted that the Morrow–Patterson points turn out to be sub-optimal in
the sense that the associated Lebesgue constants grow faster than the best possible. In
contrast, the Padua points have Lebesgue constants of optimal growth, see \cite{PD1} for a
discussion of this observation.
\vskip 2mm
\begin{remark} \label{Remark1}
We notice that if instead of \eqref{MP_curve} we sample the opposite curve \begin{equation}\label{MP_curve1}
\gamma_{n,1}(t) = \biggl(\cos((n+3)t),\cos((n+2)t)\biggl).
\end{equation}   
we obtain the set $-MP_n$ which corresponds to the upside down set $MP_n$.
\end{remark}

\begin{figure}[!h]
      \centering
        \includegraphics[ trim=0cm 0cm 0cm 0cm, clip=true,  width=0.6 \linewidth]{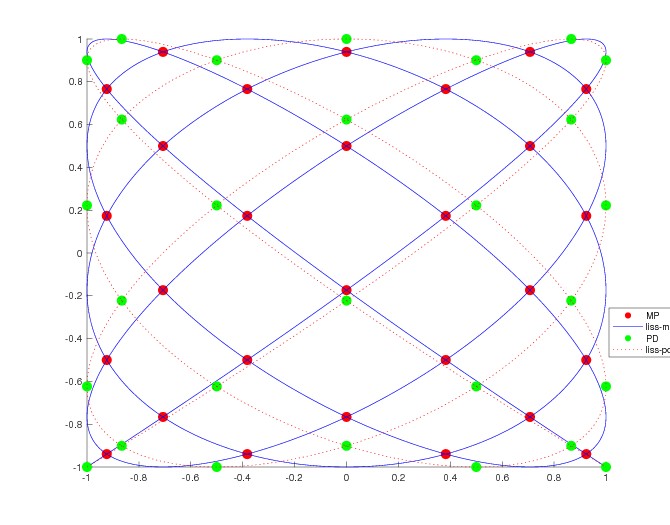} \label{fig1}
       \caption[MP]{Morrow-Patterson and Padua points for $n=6$ and the corresponding Lissajous curves}
       \label{morr_patt}
\end{figure}

\vskip 2mm
\subsection{Zeros of certain orthogonal polynomials}

The Morrow-Patterson points $MP_n$ are the zeros of the polynomials 
$$R^n_j(x_1, x_2) := U_j(x_1)U_{n-j}(x_2) + U_{n-j-1}(x_1)U_j(x_2), \;\;0 \le j \le n$$
where $U_j(x)={\sin((j+1)\theta) \over \sin\theta}, \;\; x=\cos \theta$ is the $j$th Chebyshev polynomial of the second kind, see \cite{MP}. By means of these polynomials, we can construct an orthogonal basis in the space $\mathcal{C}(Q)$ of continuous functions on the square $Q$. Indeed, the polynomials $P_{i,j}(x,y)=U_i(x)\: U_j(y)$ are orthogonal in the sense of the measure $\sqrt{1-x^2}\sqrt{1-y^2}\, dxdy$.
\vskip 1mm
Additionally, the Morrow-Patterson points are equally spaced in the sense of the {\it Dubiner metric} $\mu_Q$ in the square $Q$ (see \cite{marchi5}) where 
$$ \mu_Q({\bs x},{\bs y})=\max \left\{ |\arccos(x_1)-\arccos(y_1)|, |\arccos(x_2)-\arccos(y_2)|\right\}$$
for ${\bs x}=(x_1,x_2), \; {\bs y}=(y_1,y_2) \in Q$.

\vskip 2mm

\subsection{Points on interlacing rectangular grids}

Morrow-Patterson points are organized as two interlacing rectangular grids (cf. \cite{Floater}). This characterization can allow us to prove the insolvency of any interlacing pair of rectangular grids of points for many associated polynomial spaces. The proof uses tensor-product Newton polynomials and divided differences by reducing the problem to the solution of a sequence of smaller linear systems, similarly as done for the Padua-like points in \cite{DeMUsevich}.

The idea of representing the set $MP_n$
as interlacing rectangular grids are as follows. Consider for $k,l\ge 0$, the set of indexes
$$I_{k,l}=\{(i,j): \; 0\le i\le k;\; 0\le j\le l\} \subset \mathbb{N}^2$$
if a point set is the union of two sets
$U=\{(u_i,v_j), \; (i,j) \in I_{\mu,\nu}\}$ and $X=\{(x_i,y_j), \; (i,j)\in I_{r,s}\}$ 
it is interlaced if
$$ u_0<x_0<u_1< x_1 < \cdots \;\;or \;\; x_0<u_0<x_1<u_1< \cdots$$
and
$$ v_0<y_0<v_1< y_1 < \cdots \;\;or \;\; y_0<v_0<y_1<v_1< \cdots$$
giving $|r-\mu|\le 1, \; \; |s-\nu|\le 1$.

Then, we may associate a suitable polynomial space to interlacing rectangular grids by recalling the idea of lower sets (cf. \cite{DynFloater}). A set $L \subset \mathbb{N}^2$ is saying to be a {\it lower set} if for any $(k,l)\in L$ and $(i,j)\in \mathbb{N}^2$ such that $i\le k, \; j\le l$ we have $(i,j) \in L$.

The polynomial space associated to $L$ is then $\mathbb{P}(L)=\{x^i y^j \;:\: (i,j)\in L\}$. 
It is worth noticing that in the special case $i+j\le n$, the polynomial space reduces to $T_s:=\mathbb{P}^2_n$.

Introducing the lower sets
\begin{itemize}
    \item $K_1 \subset I_{r,s}$ such that $(0,s) \in K_1$ if $r=\mu+1$,
    \item $K_2=\{(i,j) : (r-i, s-j) \in I_{r,s} \backslash K_1\}$,
\end{itemize}
$K_2$ can be seen as the rotation of $\pi$ of $I_{r,s}\backslash K_1$.
The lower set $L$ is then 
$$\displaystyle L = I_{\mu, \nu} \cup \{(i+\mu+1,j) \;:\; (i,j)\in K_1\} \cup \{(i,j+\nu+1) \;:\; (i,j)\in K_2\},$$
which has the cardinality of $U \cup X$ that is
$$(\mu+1)(\nu+1)+(r+1)(s+1)\,.$$
This is a constructive way to prove (cf. \cite[Thm.1]{Floater})
\vskip 2mm
\begin{theorem}
Given a real function $f$ known at $U \cup X$, there is a unique interpolating polynomial $p\in\mathbb{P}(L)$, that is $$p= f_{|_{U \cup X}}\,.$$ 
\end{theorem}
\vskip 2mm
Hence, the Morrow-Patterson points can be described using these notations. In particular, referring to introduced in \cite[Sec. 2.2]{MP}, the relevant indices are $(\mu,\nu)=(n,n)$, $(r,s)=(n,n-1), \; n \ge 1$. Hence, taking $K_1=T_{n-1}$ gives $K_2=T_{n-1}$ and $$\mathbb{P}(L) =\mathbb{P}_{2n}^2\,.$$ In Figure \ref{fig2} we see the case for $n=2$.

\begin{figure}[h!] 
      \centering
        \includegraphics[clip=true,  height=0.5\linewidth, width=0.6\linewidth]{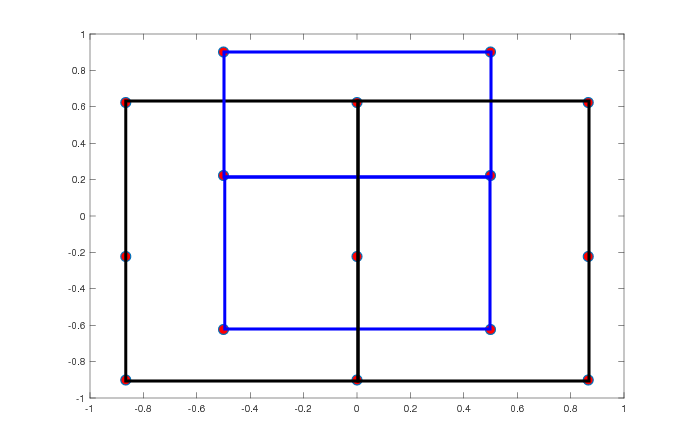} 
       \caption[MP]{The MP points as two interlacing rectangular grids with $n=2$: $K_1$ is the black rectangle, $K_2$ is the blue rectangle}
    \label{fig2}
\end{figure}

\section{Interpolation and cubature at Morrow-Patterson points}

Let $U_j$ be the Chebyshev polynomials of the second kind of degree $j=0,1,...$, i.e. 
\[ U_j(\cos \theta) := \frac{ \sin(j+1)\theta }{ \sin \theta} \ \ \ \mbox{for} \ \theta\in \mathbb{R}.\] 
We will make use of the following properties of these polynomials:
\begin{itemize}
    \item $\int_{-1}^{1} U_j(x) \, U_k(x) \sqrt{1-x^2} \, dx = {\frac {\pi }{2}}\:  \delta_{j,k}$,
    \vskip 2mm
    \item $\max_{[-1,1]} |U_j| = j+1 = U_j(1) = |U_j(-1)|,$
    \vskip 2mm
    \item the polynomials $\widehat{U_j}:=\sqrt{2 \over \pi} \, U_j$, $j=0,1,2,...$ are orthonormal in the sense of the Gegenbauer weight $\sqrt{1-x^2}\, dx$ in $[-1,1]$,
\end{itemize}

Now consider the following inner product in the space $\mathcal{C}(Q)$ 
\[ (u,v)=\int_Q u(x,y)\, v(x,y) \sqrt{1-x^2}\sqrt{1-y^2} \, dx\, dy \ \ \ \ \mbox{for} \ \ u,v\in \mathcal{C}(Q).\]
Then the polynomials 
\[ P_{i,j}(x,y):=\widehat{U_i}(x) \;\widehat{U_j}(y) , \ \ \ \ 0\le i+j\le n\]
forms an orthonormal basis in $\mathbb{P}_{n}^2$ with respect to the above inner product in $\mathcal{C}({Q})$. 
By the Morrow-Patterson result \cite{MP}, for any polynomial $p\in \mathbb{P}_{2n}^2$ the following cubature formula holds
\[  \int_{-1}^{1} \int_{-1}^{1}  p(x,y) \:  \sqrt{1-x^2} \sqrt{1-y^2}\, dx dy =  \frac{\pi^2}4 \sum_{m=1}^{n+1} \sum_{k=1}^{  \frac{n}{2} + 1 } \omega_{m,k} \: p(x_m,y_k), \]
where 
\[ \frac{1}{\omega_{m,k}} =   \sum_{i=0}^{n} \sum_{j=0}^{  n-i } {U_i}^2(x_m) \: {U_j}^2(y_k).  \]
and $(x_m, y_k)$ are the Morrow-Patterson points from $MP_n$ given in (\ref{MP_points}).
By classical methods, see, e.g. \cite{sloan}, the interpolant at $MP_n$ is given by
\begin{eqnarray*}
 L_n f (x,y) &=& \frac{\pi^2}4 \sum_{m=1}^{n+1} \sum_{k=1}^{\tfrac{n}2+1}   \omega_{m,k} \: f(x_m,y_k) \sum_{0\le i+j \le n} P_{i,j}(x_m,y_k) \: P_{i,j}(x,y)\\
 &=& \sum_{m=1}^{n+1} \sum_{k=1}^{\tfrac{n}2+1}  \omega_{m,k} \: f(x_m,y_k) \sum_{i=0}^n
\sum_{j=0}^{n-i} U_i(x_m) \: U_j(y_k) \: U_i(x) \: U_j(y)  \\
& =&  \sum_{m=1}^{n+1} \sum_{k=1}^{\tfrac{n}2+1}  \left[\omega_{m,k}  \sum_{i=0}^n
\sum_{j=0}^{n-i} U_i(x_m) \: U_j(y_k) \: U_i(x) \: U_j(y)\right] \: f(x_m,y_k) \, 
\end{eqnarray*} for $f\in \mathcal{C}(Q)$.
Hence, the {\it Lebesgue function $\lambda_n$} at the Morrow-Patterson points $MP_n$ is given by
\begin{equation} \label{leb_fun} \lambda_n(x,y)= \sum_{m=1}^{n+1} \sum_{k=1}^{\tfrac{n}2+1}  \omega_{m,k} \left| \sum_{i=0}^n
\sum_{j=0}^{n-i} U_i(x_m) \: U_j(y_k) \: U_i(x) \: U_j(y) \right| \end{equation}
and the {\it Lebesgue constant} is its sup-norm over the square $Q$
\begin{equation} \label{leb_constant} 
    \Lambda_n^{MP}:=\max\{ \lambda_n(x,y)\: : \: (x,y)\in Q\}
\end{equation}
that is equal to the norm of the operator $L_n: \mathcal{C}(Q) \ni f \mapsto L_nf \in \mathbb{P}_n^2$, i.e.
\begin{equation} \label{Leb_const_norm}
    \Lambda_n^{MP} = \sup \{ \|L_nf\|_Q \: : \: \|f\|_Q =1\}  
\end{equation} 
where in the whole paper $\|g\|_Q:=\max\{|g(x,y)| \: : \: (x,y)\in Q\}$.

\vskip2mm

\subsection{Lebesgue constant growth}

Concerning the growth of the Lebesgue constant for Morrow-Patterson points, 
Bos \cite{B01} proved that $\Lambda_n^{MP}=\mathcal{O}(n^6)$, by means of 
the bivariate Christoffel-Darboux formula. Now, interpreting interpolation 
at the Morrow-Patterson points as hyperinterpolation, in \cite{DeMSV14}
was stated the improved upper bound of order $n^3$.

\vskip 2mm
\begin{proposition}
The Lebesgue constant of the  Morrow-Patterson points has the upper bound
\begin{equation} \label{lebMP}
\Lambda_n^{MP}\leq 
\frac{1}{6\sqrt{10}}\,\sqrt{(n+1)(n+2)(n+3)(n+4)(2n^2+10n+15)}
=\mathcal{O}(n^3)\;.
\end{equation}
\end{proposition}
\vskip0.5cm
\noindent
it is worth noting that bound (\ref{lebMP}) is valid for {\em any} hyperinterpolation
operator for the product Chebyshev measure of the second
kind, as stated in \cite{DeMSV14}.

\vskip 1mm
On the other hand, the numerical evidence tells us that (\ref{lebMP}) is an 
{\it overestimate} of the order of growth, see Figure \ref{fig4}. Indeed, the values of  $\Lambda_n^{MP}$ in the range $n=2,4,6,\dots,60$ 
are well-fitted below by the quadratic polynomial $(0.7n+1)^2$ and from above by $(0.75n+1)^2$.  

\begin{figure}[h!] 
\centering
\includegraphics[scale=0.35, clip]{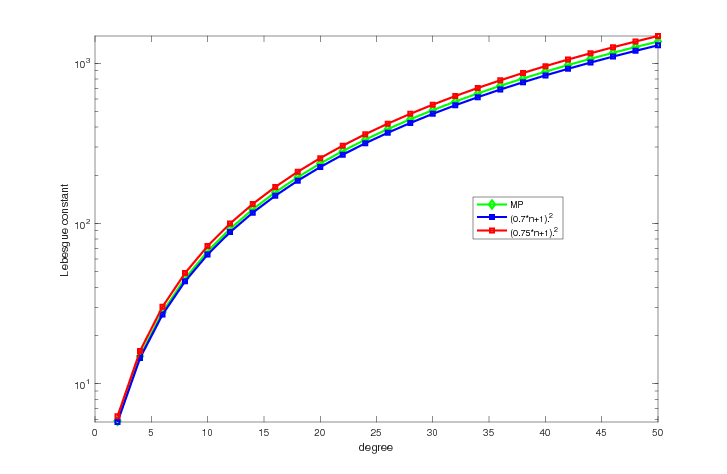} 
\caption{The numerically 
evaluated 
Lebesgue constant of interpolation at the Morrow-Patterson 
points (log scale) compared with $(0.7n+1)^2$ and $(0.75n+1)^2$.}\label{fig4}
\end{figure}

\vskip 2mm

\subsection{Cubature formula for Morrow-Patterson points} \label{Sec_cubature}

To give an easy way to calculate a cubature rule at the Morrow-Patterson points, we need a simple formula for the weights $\omega_{m,k}$, $m=1,...,n+1, \: k=1,...,\tfrac{n}2+1$ related to the points $(x_m,y_k)$ from the set $MP_n$. We introduce the following notation to state and prove these formulae and the next results. Let 
\begin{equation*} \label{MP_points_angles}
w_m := \frac{m\pi}{n+2}, \qquad v_k=v_k(m) :=
\begin{cases}
  \displaystyle \frac{2 k \pi}{n+3}                     & \text {for odd $m$,} \\ \\
  \displaystyle \frac{(2 k-1) \pi}{n+3}    & \text{for even $m$,}
  \end{cases}
\end{equation*}
so $w_m, v_k \in (0,\pi)$ for all $m,k$ and
\begin{eqnarray*}
MP_n &=& \{(x_m,y_k)\,:\, m=1,...,n+1, \: k=1,...,\tfrac{n}2+1\} \\ 
&=& \{(\cos w_m,\cos v_k)\,:\, m=1,...,n+1, \: k=1,...,\tfrac{n}2+1\}.
\end{eqnarray*}
Moreover, we will write 
\begin{equation*}
v_k':=\tfrac{2 k \pi}{n+3}, \ \ \ \ \ \ \ v_k'':=\tfrac{(2 k-1) \pi}{n+3}.
\end{equation*}

\vskip 2mm

\vskip 2mm
\begin{proposition} \label{Lemma_before_quadrature}
    Let
    \begin{equation} \label{I(i,j)}
        I(i,j):= \sum_{m=1}^{n+1} \sum_{k=1}^{\tfrac{n}{2} + 1} (1-x_m^2) (1-y_k^2) \ U_i(x_m)\: U_j(y_k) 
    \end{equation}
    \begin{equation*}
        = \sum_{m=1}^{n+1} \sum_{k=1}^{\tfrac{n}{2} + 1} \sin w_m \ \sin v_k \ \sin((i+1)w_m) \ \sin((j+1)v_k) 
    \end{equation*}
    for Morrow-Patterson points $(x_m,y_k)\in MP_n$ and non-negative integer numbers $i,j$.
    If $i$ is such that $n+2$ divides neither $i$ nor $i+2$ or $j$ is such that $n+3$ divides neither $j$ nor $j+2$, then 
    \begin{equation*}
        I(i,j)=0.
    \end{equation*}  
    Additionally, for all other cases, we have
    \begin{equation*}
        I(i,j)= 
        \begin{cases}
            \displaystyle (2C_n)^{-1}\: [(-1)^\nu + (-1)^\mu]  & \ \ ; \ \ i=\nu(n+2), \ j=\mu(n+3), \\ \\
            \displaystyle -(2C_n)^{-1}\: [(-1)^\nu + (-1)^\mu]  & \ \ ; \ \ i=\nu(n+2), \ j=\mu(n+3)-2, \\ \\
            \displaystyle -(2C_n)^{-1}\: [(-1)^\nu + (-1)^\mu]  & \ \ ; \ \ i=\nu(n+2)-2, \ j=\mu(n+3), \\ \\
            \displaystyle (2C_n)^{-1}\: [(-1)^\nu + (-1)^\mu]  & \ \ ; \ \ i=\nu(n+2)-2, \ j=\mu(n+3)-2
        \end{cases}
    \end{equation*}
    where $\mu, \nu$ are integer numbers. Any upper bound for $i, j$ is not required in the above formulae.   
\end{proposition}

\vskip 2mm

\begin{proof}      
    We will use the following two classical facts
        \begin{equation} \label{formula2}
            \sum_{\ell=0}^{M} \cos((2\ell+1) x) = \frac{\sin(2(M+1)x)}{2\sin x} ,
        \end{equation}
        \begin{equation} \label{formula3}
            \sum_{\ell=1}^{M} \cos(\ell x) = \frac{\sin\tfrac{Mx}2}{\sin\tfrac{x}2} \: \cos\tfrac{(M+1)x}2
        \end{equation}
    that hold for any positive integer $M$ and real $x$ such that $\sin x \ne 0$ or $\sin \tfrac{x}2\ne0$, respectively.

    Since $n$ is even, for any non-negative integers $i,j$ we can write
        \begin{equation*}
            I(i,j) = \sum_{{\rm odd \ } m=1}^{n+1} \sin w_m \: \sin((i+1)w_m) \sum_{k=1}^{ \tfrac{n}{2} + 1 } \sin v'_k \: \sin ((j+1)v'_k) 
        \end{equation*}
        \begin{equation*}
            + \sum_{{\rm even \ } m=2}^{n} \sin w_m \: \sin((i+1)w_m) \sum_{k=1}^{ \tfrac{n}{2} + 1 } \sin v''_k \: \sin ((j+1)v''_k) =: J_o(i)\: I'(j) + J_e(i) \: I''(j).
        \end{equation*}
    \vskip 4mm
    
    For $j$ such that $n+3$ divides neither $j$ nor $j+2$, by identity (\ref{formula2}), we get
        \begin{equation*}
            I''(j) = \tfrac12 \sum_{k=1}^{ \tfrac{n}{2} + 1 } \left[ \cos (j v''_k) - \cos ((j+2)v''_k) \right] = \tfrac12 \sum_{\ell=0}^{ \tfrac{n}{2} } \left[ \cos \tfrac{(2\ell+1)j\pi}{n+3} - \cos \tfrac{(2\ell+1)(j+2)\pi}{n+3} \right]
        \end{equation*}
        \begin{equation*}
             = \frac{\sin\tfrac{(n+2)j\pi}{n+3}}{4\sin \tfrac{j\pi}{n+3}} - \frac{\sin\tfrac{(n+2)(j+2)\pi}{n+3}}{4\sin \tfrac{(j+2)\pi}{n+3}} = \tfrac14 (-1)^{j+1} - \tfrac14 (-1)^{j+3} =0.
        \end{equation*}
    Moreover, if $j=\mu(n+3)$ or $j=\mu(n+3)-2$ for some integer $\mu$, then we can easily check the identities
        \begin{equation*}
            I''(\mu(n+3)) = (-1)^\mu \: \tfrac{n+3}4 \ \ \ \ \ \ \ \mbox{and} \ \ \ \ \ \ \ \ I''(\mu(n+3)-2) = -(-1)^\mu \: \tfrac{n+3}4. 
        \end{equation*}

    The sum with $v'_k$ can be evaluated in a similar way, using formula (\ref{formula3}), for $j$ such that $n+3$ divides neither $j$ nor $j+2$, we have   
        \begin{equation*}
            I'(j)=\sum_{k=1}^{ \tfrac{n}{2} + 1 } \sin v'_k \: \sin ((j+1)v'_k) = \tfrac12 \sum_{k=1}^{ \tfrac{n}{2} + 1 } \left[ \cos \tfrac{2jk\pi}{n+3} - \cos \tfrac{2(j+2)k\pi }{n+3} \right] 
        \end{equation*}
        \begin{equation*}
            = \tfrac12 \left[ \frac{\sin ((\tfrac{n}2 +1)\tfrac{j\pi}{n+3}) \ \cos ((\tfrac{n}2 +2)\tfrac{j\pi}{n+3}) }{\sin \tfrac{j\pi}{n+3}} - \frac{\sin ((\tfrac{n}2 +1)\tfrac{(j+2)\pi}{n+3}) \ \cos ((\tfrac{n}2 +2)\tfrac{(j+2)\pi}{n+3}) }{\sin \tfrac{(j+2)\pi}{n+3}} \right] 
        \end{equation*}
        \begin{equation*}
            = \tfrac12 \left[ \frac{\sin (j\pi) - \sin \tfrac{j\pi}{n+3} }{ 2 \sin \tfrac{j\pi}{n+3}} - \frac{\sin ((j+2)\pi) - \sin \tfrac{(j+2)\pi}{n+3} }{ 2\sin \tfrac{(j+2)\pi}{n+3}} \right] = 0
        \end{equation*}
        and 
        \begin{equation*}
            I'(\mu(n+3))=\tfrac{n+3}4, \ \ \ \ \ \ \ \ \ \ \ \ \ I'(\mu(n+3)-2)=-\tfrac{n+3}4
        \end{equation*}
        for integer number $\mu$.

    As above, we can calculate  
        \begin{equation*}
            J_o(i) = 0 \ \ \ \ \ \ \ \ \mbox{and} \ \ \ \ \ \ \ \ \ J_e(i) = 0
        \end{equation*}
    for $i$ such that $n+2$ does not divide either $i$ or $i+2$. If $i=\nu (n+2)$ or $i=\nu(n+2) -2$ for some integer $\nu$, then 
        \begin{equation*}
            J_o(\nu (n+2)) = (-1)^\nu \tfrac{n+2}4, \ \ \ \ \ J_e(\nu(n+2)) = \tfrac{n+2}4,
        \end{equation*}
        \begin{equation*}
            J_o(\nu (n+2)-2) = -(-1)^\nu \tfrac{n+2}4, \ \ \ \ \ J_e(\nu(n+2)-2) = -\tfrac{n+2}4.
        \end{equation*}
\vskip 3mm
    The above cases may be summarized by saying that 
        \begin{equation*}
            I(i,j)=J_o(i)\: I'(j) + J_e(i) \: I''(j) = 0
        \end{equation*}
    if  $i$ is such that $n+2$ divides neither $i$ nor $i+2$ or $j$ is such that $n+3$ divides neither $j$ nor $j+2$. Additionally, we can calculate values of $I(i,j)$ in all other cases:
        \begin{equation*}
            I(\nu(n+2),\mu(n+3)) = I(\nu(n+2)-2,\mu(n+3)-2) = \tfrac{(n+2)(n+3)}{16} \: [(-1)^\nu + (-1)^\mu],  
        \end{equation*}
        \begin{equation*}
            I(\nu(n+2),\mu(n+3)-2) = I(\nu(n+2)-2,\mu(n+3)) = - \tfrac{(n+2)(n+3)}{16} \: [(-1)^\nu + (-1)^\mu].  
        \end{equation*} 
\end{proof}

\vskip 2mm
\begin{theorem} \label{cub_formula}
The cubature formula for the Morrow-Patterson points $MP_n$ on the square $Q$ is given by 
\begin{equation} \label{quadrature}
\frac4{\pi^2}\int\limits_Q p(x,y) \:  \sqrt{1-x^2} \sqrt{1-y^2}\, dx dy = C_n \sum_{m=1}^{n+1} \sum_{k=1}^{\tfrac{n}{2} + 1} (1-x_m^2) (1-y_k^2) \ p(x_m,y_k) 
\end{equation}
$$= C_n \sum_{m=1}^{n+1} \sum_{k=1}^{  \tfrac{n}{2} + 1 } \sin^2w_m\: \sin^2v_k \ p(\cos w_m,\cos v_k) 
$$
\vskip 4mm
\noindent
for any polynomial $p\in \mathbb{P}_{2n}^2$. 
Consequently,
\begin{equation} \label{waga}
\omega_{m,k}= C_n \: (1-x_m^2)(1-y_k^2).
\end{equation}
The projection $\mathcal{C}(Q) \longrightarrow \mathbb{P}_n^2$ with respect to the discrete inner product 
$$(u,v)_\star:=\tfrac{\pi^2}4 \sum_{m=1}^{n+1} \sum_{k=1}^{\tfrac{n}{2}+1} w_{m,k} \: u(x_m,y_k)\: v(x_m,y_k)$$ 
is an interpolation operator with nodes at Morrow-Patterson points $MP_n$, i.e. for 
\begin{equation*}
 L_n f (x,y) = C_n \sum_{m=1}^{n+1} \sum_{k=1}^{\tfrac{n}2+1} (1-x_m^2)(1-y_k^2) \: f(x_m,y_k) \sum_{i=0}^n \sum_{j=0}^{n-i} U_i(x_m) \: U_j(y_k) \: U_i(x) \: U_j(y)  \\ 
\end{equation*}
the interpolation conditions $L_n f (x_m,y_k)=f(x_m,y_k)$, $m=1,...,n+1$, $k=1,...,\tfrac{n}2+1$, are satisfied.
The Lebesgue function corresponding to the Morrow-Patterson points is given by
\begin{equation} \label{leb_fun_exact} 
\lambda_n(x,y) = C_n\sum_{m=1}^{n+1} \sum_{k=1}^{\tfrac{n}2+1}  (1-x_m^2)(1-y_k^2) \left| \sum_{i=0}^n
\sum_{j=0}^{n-i} U_i(x_m) \: U_j(y_k) \: U_i(x) \: U_j(y) \right|.
\end{equation}
\vskip2mm    
\end{theorem}
\vskip 3mm

\begin{proof}
    First, observe that it is sufficient to prove the quadrature rule (\ref{quadrature}) only for polynomials $p(x,y)=U_i(x)\,U_j(y)$ and $0\le i+j\le 2n$. Notice that the left-hand side of (\ref{quadrature}) equals $0$ when $i\ne j$ or $1$ if $(i,j)=(0,0)$. Therefore, to prove (\ref{quadrature}), it is sufficient to show that
        \begin{equation} \label{2_to_prove}
          I(i,j)=\sum_{m=1}^{n+1} \sum_{k=1}^{  \tfrac{n}{2} + 1 } \sin^2w_m\: \sin^2v_k \ U_i(\cos w_m) \, U_j(\cos v_k) =0, \; \; i\ne j,  
        \end{equation}
        \begin{equation} \label{1_to_prove}
            I(0,0)=\sum_{m=1}^{n+1} \sum_{k=1}^{  \tfrac{n}{2} + 1 } \sin^2w_m\: \sin^2v_k = {1 \over C_n}.
        \end{equation}
    where $I(i,j)$ is defined by formula (\ref{I(i,j)}). So, taking $\mu=\nu=0$ in Proposition \ref{Lemma_before_quadrature}, we can see that formula (\ref{1_to_prove}) holds. 

    For fixed $i\in\{0,...,2n\}$ the number $n+2$ divides $i$ or $i+2$ only if $i=0$, $i=n$ or $i=n+2$. Considering $j\in\{0,...,2n\}$ such that $n+3$ divide $j$ or $j+2$, we get $j=0$, $j=n+1$ and $j=n+3$. We are interested only in $0<i+j\le 2n$, so $(i,j)\in\{ (0,n+1),(0,n+3),(n,0),(n+2,0)\}$.  By Proposition \ref{Lemma_before_quadrature}, $I(0,n+1)=I(0,n+3)=I(n,0)=I(n+2,0)=0=I(i,j)$ for all other $(i,j)$ such that $0<i+j\le 2n$ and identity (\ref{2_to_prove}) is proved.

    The statements concerning the projection and the interpolation operators can be derived from \cite[Lemma 3]{sloan}. Formula (\ref{leb_fun_exact}) is a simple consequence of (\ref{leb_fun}) and (\ref{waga}).
\end{proof}

\vskip 2mm

\begin{remark}
    Cubature formula (\ref{quadrature}) for Morrow-Patterson points $MP_n$ holds for all polynomials $p\in \mathbb{P}_{2n}^2$, while an analogous rule for Padua points $Pad_n$ requires an additional assumption on polynomials (and is not satisfied, e.g. for $p(x,y)=T_{2n}(y)$), see \cite[Th. 1]{PD1} which is saying that 
    \vskip 2mm
    \noindent {\tt If $q\in\mathbb{P}_{2n}^2$ satisfies the condition 
    \vskip -1mm 
    \begin{equation} \label{ass_Padua}
        \int_Q q(x,y)\: T_{2n}(y) \tfrac{dxdy}{\sqrt{1-x^2}\sqrt{1-y^2}}=0
    \end{equation}
    \vskip -1mm \noindent then \vskip -1mm
    \[ \int_Q q(x,y)\: \tfrac{dxdy}{\sqrt{1-x^2}\sqrt{1-y^2}}= \pi^2 \sum_{a\in Pad_n} \tfrac{\omega_a}{n(n+1)} \: q(a) \] 
    \vskip -1mm
    \noindent where $\omega_a=\tfrac12, \, 1$ or  2 for vertex points, edge points or interior points of $Q$, respectively. } 
    \vskip 2mm
    \noindent Because of assumption (\ref{ass_Padua}), the projection with respect to the corresponding discrete inner product is not an interpolation operator for Padua points.
\end{remark}

\vskip 3mm
Despite the above comment, it is worth noting that the quadrature (\ref{quadrature}) can be derived from \cite[Th. 1]{PD1}. Indeed, consider polynomial $p\in\mathbb{P}_{2n}^2$ and take $$q(x,y)=p(x,y)\,(1-x^2)(1-y^2)\in \mathbb{P}_{2(n+2)}^2.$$ Since \ $p(x,y)\: (1-y^2) = \sum_{j=0}^{2n+2} a_j(x) \: T_j(y)$ \ for some $a_j\in\mathbb{P}_{2n+2-j}^2$, we have    
\[ \int_Q q(x,y)\: T_{2n+4}(y) \tfrac{dxdy}{\sqrt{1-x^2}\sqrt{1-y^2}}= \int_Q p(x,y)\: (1-y^2)\: T_{2n+4}(y) \: \sqrt{1-x^2}\: \tfrac{dxdy}{\sqrt{1-y^2}} \]
\[ = \sum_{j=0}^{2n+2} \int_{-1}^1 a_j(x) {\sqrt{1-x^2} \int_{-1}^1 T_j(y)\: T_{2n+4}(y) \: \sqrt{1-y^2}} {dydx} =0\]
because of the orthogonality of Chebyshev polynomials of the first kind. Consequently,
\[ \int_Q p(x,y)\: {\sqrt{1-x^2}\sqrt{1-y^2}} \: {dxdy} = \int_Q q(x,y)\: \tfrac{dxdy}{\sqrt{1-x^2}\sqrt{1-y^2}}= \pi^2 \sum_{a\in Pad_{n+2}} \tfrac{\omega_a C_n}{8} \: q(a) \] 
\[ = \pi^2 \sum_{(x,y)\in Pad_{n+2}} \tfrac{\omega_{(x,y)} C_n}{8} \ p(x,y) \: (1-x^2)(1-y^2) \] \[ = \pi^2 \sum_{(x,y)\in MP_{n}} {C_n \over 4} \ p(x,y) \: (1-x^2)(1-y^2)\]
and
\[ \frac4{\pi^2}\int\limits_Q p(x,y) \:  \sqrt{1-x^2} \sqrt{1-y^2}\, dx dy = C_n \sum_{(x,y)\in MP_{n}}  \ p(x,y) \ (1-x^2)(1-y^2) \] that is equivalent to (\ref{quadrature}).

\vskip 2mm

In the last section we will use the following fact.

\begin{lemma} \label{symLf}
    Let $\lambda_n$ be the  {\it Lebesgue function} at the Morrow-Patterson points. Then $\lambda_n(x,y)=\lambda_n(-x,y)$ for all $(x,y) \in Q$.
\end{lemma}

\begin{proof}
    Since $n$ is even, if $m$ is even then $n+2 - m$ is even, and if $m$ is odd then $n+2 - m$ is odd. Notice that,
    \begin{align*}
        &\{(-x_m,y_k)\,:\, m=1,...,n+1, \: k=1,...,\tfrac{n}2+1\}  \\
         &=\left\{\left(\cos\Big(\pi - \frac{m \pi}{n+2}\Big),y_k\right)\,:\, m=1,...,n+1, \: k=1,...,\tfrac{n}2+1 \right\}
         \\
         &=\left\{\left(\cos\Big( \frac{n+2 -m }{n+2} \pi\Big),y_k\right)\,:\, m=1,...,n+1, \: k=1,...,\tfrac{n}2+1 \right\}
         \\
         &=\left\{\left(\cos \frac{m \pi}{n+2},y_k\right)\,:\, m=1,...,n+1, \: k=1,...,\tfrac{n}2+1 \right\}=MP_n
    \end{align*}
   Then, by Theorem \ref{cub_formula}, 
   \begin{align*}
       \lambda_n(-x,y) &= C_n\sum_{m=1}^{n+1} \sum_{k=1}^{\tfrac{n}2+1}  (1-x_m^2)(1-y_k^2) \left| \sum_{i=0}^n
\sum_{j=0}^{n-i} U_i(x_m) \: U_j(y_k) \: U_i(-x) \: U_j(y) \right|
         \\ &= C_n\sum_{(x_m,y_k) \in MP_n}  (1-x_m^2)(1-y_k^2) \left| \sum_{i=0}^n
\sum_{j=0}^{n-i} U_i(x_m) \: U_j(y_k) \: U_i(-x) \: U_j(y) \right|
         \\ &= C_n\sum_{(-x_m,y_k) \in MP_n}  (1-x_m^2)(1-y_k^2) \left| \sum_{i=0}^n
\sum_{j=0}^{n-i} U_i(-x_m) \: U_j(y_k) \: U_i(-x) \: U_j(y) \right|
        \\&=C_n\sum_{m=1}^{n+1} \sum_{k=1}^{\tfrac{n}2+1}  (1-x_m^2)(1-y_k^2) \left| \sum_{i=0}^n
\sum_{j=0}^{n-i} U_i(-x_m) \: U_j(y_k) \: U_i(-x) \: U_j(y) \right|.
   \end{align*}
  Since $U_i(-x_m)  U_i(-x) = U_i(x_m)  U_i(x)$ we get the claim.
\end{proof}

\vskip 2mm

\section{Optimal admissible meshes at Morrow-Patterson points}

For a fixed integer $\nu>0$, consider the following polynomial meshes in the square $Q$: 
\begin{equation*}
    \mathcal{A}_{\nu}:=
        \begin{cases}
            \displaystyle MP_{\nu}\cup\{(1,1),(-1,1)\}  & \ \ ; \ \ \nu \ \mbox{is even,} \\ \\
            \displaystyle MP_{\nu}\cup\{(1,1),(1,-1)\}  & \ \ ; \ \ \nu \ \mbox{is odd,}
        \end{cases}  
\end{equation*}
\begin{equation*}
    \mathcal{B}_{\nu}:=
        \begin{cases}
            \displaystyle MP_{\nu}\cup\{(\cos\tfrac\pi{\nu+2}, \cos\tfrac\pi{\nu+3}),(\cos\tfrac{(\nu+1)\pi}{\nu+2}, \cos\tfrac\pi{\nu+3})\}  & \ \ ; \ \ \nu \ \mbox{is even,} \\ \\
            \displaystyle MP_{\nu}\cup\{(\cos\tfrac\pi{\nu+2}, \cos\tfrac\pi{\nu+3}),(\cos\tfrac\pi{\nu+2}, \cos\tfrac{(\nu+2)\pi}{\nu+3})\}  & \ \ ; \ \ \nu \ \mbox{is odd.}
        \end{cases}  
\end{equation*}
\vskip 2mm \noindent
We see that the mesh $\mathcal{B}_{\nu}$ is a subset of the rectangle $[\cos\tfrac\pi{\nu+2}, \cos\tfrac{(\nu+1)\pi}{\nu+2}] \times [\cos\tfrac{\pi}{\nu+3}, \cos\tfrac{(\nu+2)\pi}{\nu+3}]$ which is  a smaller rectangular that the one where $\mathcal{A}_{\nu}$ is contained. 

As usual, we denote by $\lceil x \rceil$ the ceiling of the real number $x$. We prove that the Morrow-Patterson points are optimal admissible meshes in $Q$ as stated in the following theorem.

\vskip 3mm

\begin{theorem}
    Let $n$ be a positive integer and $\mu >2$. 
    Then, for any polynomial $p\in\mathbb{P}_n^2$ we have
    \begin{equation} \label{adm_meshes}
        \| p \|_Q \ \le \ \tfrac1{\cos \tfrac\pi\mu} \ \min \{ \|p\|_{\mathcal{A}_{\lceil\mu n\rceil}},\ \|p\|_{\mathcal{B}_{\lceil\mu n\rceil}}, \ \|p\|_{MP_{\lceil 2\mu n\rceil}}\}.
    \end{equation}
    In other words, ${\mathcal{A}_{\lceil\mu n\rceil}}, \ {\mathcal{B}_{\lceil\mu n\rceil}}, \ {MP_{\lceil 2\mu n\rceil}}$ are optimal admissible meshes for $\mathbb{P}_n^2$ in the square $Q$.
    Moreover,
    \begin{equation} \label{Leb_const_mesh}
        \Lambda_n^S \ \le \ \tfrac1{\cos \tfrac\pi\mu} \ \min\{\:\|\lambda_n^S\|_{\mathcal{A}_{\lceil\mu n\rceil}}, \ \|\lambda_n^S\|_{\mathcal{B}_{\lceil\mu n\rceil}}, \ \|\lambda_n^S\|_{MP_{\lceil 2\mu n\rceil}}\}
    \end{equation}
    where $\Lambda_n^S$ and $\lambda_n^S$ are the Lebesgue constant and function, respectively, corresponding to any set $S\subset Q$ of $\tfrac{(n+1)(n+2)}2$ nodes which is unisolvent for $\mathbb{P}_n^2$.
\end{theorem}
\vskip 2mm
\begin{proof}
    We start with an evaluation of the covering radius with respect to the Dubiner metric
    \begin{equation*}
        \varrho_Q(\mathcal{A}):=\sup_{(x,y)\in Q} \inf_{(a,b)\in\mathcal{A}} d_Q((x,y),(a,b))
    \end{equation*}
    where $d_Q$ is the Dubiner distance between two points from the square $Q$, i.e.
    \begin{eqnarray*}
        d_Q((x,y),(a,b)) &:=& \sup\{ \tfrac1{{\rm deg}\,p}\: |\cos^{-1}(p(x,y))-\cos^{-1}(p(a,b))|\: : \\ 
        & & \ \ \ \ \ \ \ \ \ \ \ \ \ \ \ \: p\in\mathbb{P}^2, \: {\rm deg}\,p\ge 1, \: \|p\|_Q \le 1 \} \\
        &=& \max\{ |\cos^{-1}(x)-\cos^{-1}(a)|,\: |\cos^{-1}(y)-\cos^{-1}(b)| \},
    \end{eqnarray*}
    see \cite{BLW2004,marchi5}. First consider 
    \begin{eqnarray*}
        \varrho_Q(MP_{\nu}) &=& \max_{(x,y)\in Q} \inf_{(a,b)\in MP_{\nu}} \max\{ |\cos^{-1}(x)-\cos^{-1}(a)|,\: |\cos^{-1}(y)-\cos^{-1}(b)| \} \\
        &=& \max_{\phi, \theta \in [0,\pi]} \inf_{\substack{m=1,...,\nu+1, \\ k=1,...,\nu/2+1}}
        \max\{ |\phi-w_m|,\: |\theta-v_k| \} 
    \end{eqnarray*}
    where $\nu$ is a fixed positive integer. For odd $\nu$ the points $\{(w_m,v_k)\}_{m,k}$ form a grid 
    \begin{eqnarray*}
        (\tfrac\pi{\nu+2},\tfrac{2\pi}{\nu+3}), \ \ \ \ (\tfrac\pi{\nu+2},\tfrac{4\pi}{\nu+3})\ \ , &...& , \ \ (\tfrac\pi{\nu+2},\tfrac{(\nu+1)\pi}{\nu+3}), \\ (\tfrac{2\pi}{\nu+2},\tfrac{\pi}{\nu+3}), \ \ \ \  (\tfrac{2\pi}{\nu+2},\tfrac{3\pi}{\nu+3})\ \ , &...& , \ \ (\tfrac{2\pi}{\nu+2},\tfrac{(\nu+2)\pi}{\nu+3}), \\
        & ...& \\
        (\tfrac{(\nu+1)\pi}{\nu+2},\tfrac{\pi}{\nu+3}), \: (\tfrac{(\nu+1)\pi}{\nu+2},\tfrac{3\pi}{\nu+3}), &...& , (\tfrac{(\nu+1)\pi}{\nu+2},\tfrac{(\nu+2)\pi}{\nu+3})
    \end{eqnarray*}                
    in the square $[0,\pi]^2$. Observe that if there exists at least one point $(w_m,v_k)$ in a rectangle $R=R(\ell_1,\ell_2)=[\ell_1 \tfrac\pi{\nu+2},(\ell_1+1)\tfrac\pi{\nu+2}]\times [\ell_2 \tfrac\pi{\nu+3},(\ell_2+1)\tfrac\pi{\nu+3}]$ for some $\ell_1\in \{0,...,\nu+1\}, \: \ell_2\in\{0,...,\nu+2\}$, then 
    \begin{equation*}
        \max_{(\phi,\theta)\in R} \inf_{m,k} \max\{ |\phi-w_m|,\: |\theta-v_k| \} = \frac\pi{\nu+2}.
    \end{equation*}
    
    In the case of odd $\nu$, we can find only two rectangles of this type without points $(x_m,v_k)$, i.e. \ $[0,\tfrac\pi{\nu+2}] \times [0,\tfrac\pi{\nu+3}]$ and $[0,\tfrac\pi{\nu+2}] \times [\tfrac{(\nu+2)\pi}{\nu+3},\pi]$. After adding points $\{(0,0),(0,\pi)\}$ or $\{(\tfrac\pi{\nu+2},\tfrac\pi{\nu+3}), (\tfrac\pi{\nu+2}, \tfrac{(\nu+2)\pi}{\nu+3})\}$ to $\{(w_m,v_k)\}_{m,k}$ we obtain meshes for which the covering radius in $Q$ is equal to $\frac\pi{\nu+2}$. In this way we get the meshes $MP_{\nu}\cup\{(1,1),(1,-1)\}$ and $MP_{\nu}\cup\{(\cos\tfrac\pi{\nu+2}, \cos\tfrac\pi{\nu+3}),(\cos\tfrac\pi{\nu+2}, \cos\tfrac{(\nu+2)\pi}{\nu+3})\}$. 
    
    The case of even $\nu$ is similar.
    
    Consequently, 
    \begin{equation*}
        \varrho_Q(\mathcal{A}_\nu)= \varrho_Q(\mathcal{B}_\nu) = \frac\pi{\nu+2}.
    \end{equation*}
    Take $\nu={\lceil\mu n\rceil}$. Then 
    \begin{equation*}
        \varrho_Q(\mathcal{A}_{\lceil\mu n\rceil}) = \varrho_Q(\mathcal{B}_{\lceil\mu n\rceil}) = \frac\pi{\lceil\mu n\rceil+2} \le \frac\pi{\mu n}
    \end{equation*}
    and thanks to \cite[Proposition 1]{PV}, for any polynomial $p\in\mathbb{P}_n^2$ the inequalities
    \begin{equation*}
        \| p \|_Q \ \le \ \tfrac1{\cos \tfrac\pi\mu} \ \|p\|_{\mathcal{A}_{\lceil\mu n\rceil}},\ \ \ \ \ \ \ \ \ \ \ \ \ \| p \|_Q \ \le \ \tfrac1{\cos \tfrac\pi\mu} \ \|p\|_{\mathcal{B}_{\lceil\mu n\rceil}}
    \end{equation*}
    hold. If we do not add supplementary points to $MP_\nu$, then 
    \begin{equation*}
        \varrho_Q(MP_\nu)=\frac{2\pi}{\nu+3}
    \end{equation*}
    and for $\nu={\lceil 2\mu n\rceil}$ we get
    \begin{equation*}
        \varrho_Q(MP_{\lceil 2\mu n\rceil})=\frac{2\pi}{{\lceil 2\mu n\rceil}+3} \le \frac\pi{\mu n}.
    \end{equation*}
    Consequently, inequality (\ref{adm_meshes}) is proved. Optimality of these meshes is evident, because \ card$\,MP_\nu=\tfrac{(\nu+1)(\nu+2)}2 = \mathcal{O}(\nu^2)=\mathcal{O}(n^2)$. 

    To show an estimate of the Lebesgue constant $\Lambda_n^S$ (c.f. \cite{BKSV}), fix $f\in \mathcal{C}(Q)$. Then the interpolation polynomial $L_n^S f\in \mathbb{P}_n^2$ and we can apply inequality (\ref{adm_meshes})
    \begin{equation*}
        \| L_n^S f \|_Q \ \le \ \tfrac1{\cos \tfrac\pi\mu} \ \min \{ \|L_n^S f\|_{\mathcal{A}_{\lceil\mu n\rceil}},\ \|L_n^S f\|_{\mathcal{B}_{\lceil\mu n\rceil}}, \ \|L_n^S f\|_{MP_{\lceil 2\mu n\rceil}}\}.
    \end{equation*}
    For any set $\mathcal{T}\subset Q$ we have
    \begin{equation*} 
        \| L_n^S f\|_\mathcal{T} \le \| f\|_Q \| \lambda_n^S \|_\mathcal{T}.
    \end{equation*}
    Since the Lebesgue constant $\Lambda_n^S$ is the norm of the linear operator $L_n^S:\mathcal{C}(Q) \rightarrow \mathbb{P}_n^2$ with respect to the sup-norms, i.e. 
    $$  \Lambda_n^S=\sup\{ \|L_n^S f\|_Q\: : \: \|f\|_Q=1 \},$$ 
    the proof is complete.
\end{proof}    

\vskip 2mm

\section{Auxiliary results}
This section is devoted to some lemmas that will be used to prove the main result.
\begin{lemma}
For every $l \in \mathbb{N}$ and each $\alpha, \beta \in \mathbb{R}$ such that $\sin \frac{\alpha \pm \beta }{2} \neq 0$, we have
\begin{align}\label{lem1}
\sum_{i=0}^{ l }   \sin(i+1) \alpha \sin(i+1) \beta = \frac{ \sin\frac{(2l+3)(\alpha - \beta)}{2}}{ 4 \sin  \frac{\alpha - \beta}{2} } - \frac{ \sin\frac{(2l+3)(\alpha + \beta)}{2}}{ 4 \sin  \frac{\alpha + \beta}{2} }
\end{align}
\end{lemma}
\begin{proof}
  Since $\sin a \cdot \sin b = \frac{1}{2} (\cos(a-b)-\cos(a+b))$, we conclude that
\begin{align*}
   \sum_{i=0}^{ l }   \sin(i+1) \alpha \sin(i+1) \beta = \frac{1}{2} \sum_{i=0}^{ l } \cos(i+1)(\alpha - \beta) - \frac{1}{2} \sum_{i=0}^{l } \cos(i+1)(\alpha + \beta).
\end{align*}
The well-known formula (\ref{formula3}) leads to
\begin{align*}
 \sum_{i=0}^{ l }   \sin(i+1) \alpha \sin(i+1) \beta = &\frac{ \sin\frac{(l+1)(\alpha - \beta)}{2}}{ 2\sin  \frac{\alpha - \beta}{2} } \cos\frac{(l+2)(\alpha - \beta)}{2}    \\ &- \frac{ \sin\frac{(l+1)(\alpha + \beta)}{2}}{ 2\sin  \frac{\alpha + \beta}{2} } \cos\frac{(l+2)(\alpha + \beta)}{2}
    \end{align*} 
Then the trigonometric identity $\sin a \cdot \cos b = \frac{1}{2} ( \sin(a+b) + \sin(a-b) )$  gives the desired formula. 
\end{proof}

Similarly, starting with $\cos a \cos b = \frac{1}{2}( \cos(a-b) + \cos(a+b) )  $, one obtains 
\begin{lemma}
For every $l \in \mathbb{N}$ and each $\alpha, \beta \in \mathbb{R}$ such that $\sin \frac{\alpha \pm \beta }{2} \neq 0$, we have
\begin{align}\label{lem2}
\sum_{i=0}^{ l }   \cos(i+1) \alpha \cos(i+1) \beta =\frac{ \sin\frac{(2l+3)(\alpha - \beta)}{2}}{ 4 \sin  \frac{\alpha - \beta}{2} } + \frac{ \sin\frac{(2l+3)(\alpha + \beta)}{2}}{ 4 \sin  \frac{\alpha + \beta}{2} } -\frac{1}{2}
\end{align}
\end{lemma}
\noindent To prove further formulae we will also need the following one.
\begin{lemma}
For every $l \in \mathbb{N}$ and each $\alpha, \beta \in \mathbb{R}$ such that $\sin \frac{\alpha \pm \beta }{2} \neq 0$, we have
\begin{align}\label{lem3}
 \sum_{i=0}^{ l }   \sin(i+1) \alpha \cos(i+1) \beta =\frac{ \cos \frac{\alpha - \beta}{2} - \cos\frac{(2l+3)(\alpha - \beta)}{2}}{ 4 \sin  \frac{\alpha - \beta}{2} } + \frac{ \cos \frac{\alpha + \beta}{2} - \cos\frac{(2l+3)(\alpha + \beta)}{2}}{  4 \sin  \frac{\alpha + \beta}{2} }  
\end{align}
\end{lemma}
\begin{proof}
Since $\sin a \cdot \cos b = \frac{1}{2} ( \sin(a-b) + \sin(a+b) )$, we conclude that
\begin{align*}
   \sum_{i=0}^{ l }   \sin(i+1) \alpha \cos(i+1) \beta = \frac{1}{2} \sum_{i=0}^{ l } \sin(i+1)(\alpha - \beta) + \frac{1}{2} \sum_{i=0}^{l } \sin(i+1)(\alpha + \beta).
\end{align*}
Then we use the following well-known identity
\begin{align}\label{sin}
  \sum_{\ell=1}^{M} \sin(\ell x) = \frac{\sin\tfrac{Mx}2}{\sin\tfrac{x}2} \: \sin\tfrac{(M+1)x}2,
\end{align}
which leads to
\begin{align*}
 \sum_{i=0}^{ l }   \sin(i+1) \alpha \sin(i+1) \beta = &\frac{ \sin\frac{(l+1)(\alpha - \beta)}{2}}{ 2\sin  \frac{\alpha - \beta}{2} } \sin\frac{(l+2)(\alpha - \beta)}{2}    \\ &+ \frac{ \sin\frac{(l+1)(\alpha + \beta)}{2}}{ 2\sin  \frac{\alpha + \beta}{2} } \sin\frac{(l+2)(\alpha + \beta)}{2}
    \end{align*} 
 Then the trigonometric identity $\sin a \cdot \sin b = \frac{1}{2} ( \cos(a-b) - \cos(a+b) )$  gives the desired formula.    
\end{proof}

\begin{lemma}\label{lem4}
Let $l,\alpha,\beta$ be as above. Assume further that  $\sin \beta \neq 0$. Then 
\begin{align*}
\frac{\sin \alpha }{ \sin \beta} & \sum_{i=0}^{ l }   \sin(i+1) \alpha \sin(i+1) \beta 
\\=& \frac{ \cos \frac{\alpha}{2} \cos \frac{\beta}{2} }{ 2\sin  \beta} \left( \sin\frac{(2l+3)(\alpha - \beta)}{2} -  \sin\frac{(2l+3)(\alpha + \beta)}{2} \right)
\\ &+\frac{\cos \frac{\alpha}{2}      }{ 4\cos \frac{\beta}{2}  } \left(  \frac{\cos \frac{\alpha - \beta}{2}}{ \sin \frac{\alpha - \beta}{2}}   \sin\frac{(2l+3)(\alpha - \beta)}{2} 
   + \frac{\cos \frac{\alpha + \beta}{2}}{ \sin \frac{\alpha + \beta}{2}} \sin\frac{(2l+3)(\alpha + \beta)}{2} \right).
\end{align*}
\end{lemma}
\begin{proof}
By using (\ref{lem1}) and $\sin(a \pm b)=\sin a \cos b \pm \sin b \cos a$, we have
\begin{align*}
\frac{\sin \alpha }{ \sin \beta} & \sum_{i=0}^{ l }   \sin(i+1) \alpha \sin(i+1) \beta  = \frac{2 \sin \frac{\alpha}{2} \cos \frac{\alpha}{2}}{ \sin \beta}  \sum_{i=0}^{ l }   \sin(i+1) \alpha \sin(i+1) \beta
\\=& \frac{ \cos \frac{\alpha }{2} \cos \frac{\beta}{2} \sin\frac{(2l+3)(\alpha - \beta)}{2}}{ 2\sin  \beta} +\frac{\cos \frac{\alpha}{2} \sin \frac{\beta}{2} \cos   \frac{\alpha - \beta}{2}  \sin\frac{(2l+3)(\alpha - \beta)}{2}}{ 2\sin \beta \sin  \frac{\alpha - \beta}{2} }
\\ &- \frac{ \cos \frac{\alpha}{2} \cos \frac{\beta}{2} \sin\frac{(2l+3)(\alpha + \beta)}{2}}{ 2\sin  \beta} +\frac{\cos \frac{\alpha}{2} \sin \frac{\beta}{2} \cos   \frac{\alpha + \beta}{2}  \sin\frac{(2l+3)(\alpha + \beta)}{2}}{ 2\sin \beta \sin  \frac{\alpha + \beta}{2} }
\end{align*}
Now the double-angle identity $\sin (2x) = 2 \cos x \sin x$ gives the desired formula. 
\end{proof}

\begin{lemma}\label{lem5}
For every natural number $n$ and each $x,y,z \in \mathbb{R}$ so that $\sin \frac{x+y\pm2z}{2} \neq 0$ and $\sin \frac{x-y\pm2z}{2} \neq 0$, the following formula holds
\begin{align*}
\Upsilon&_n(x,y,z):=\sum_{j=0}^{ n }    \sin(j+1) x \sin(j+1) y \sin (2n-2j+3) z  
\\ &= \frac{\cos \frac{x-y + (8+4n)z}{2} }{8\sin \frac{x-y-2z}{2}} - \frac{\cos \frac{(2n+3)(x-y) +4z }{2}}{8\sin \frac{x-y-2z}{2}} +  \frac{\cos \frac{(2n+3)(x-y) - 4z}{2} }{8\sin \frac{x-y+2z}{2}} - \frac{\cos \frac{x-y - (8+4n)z}{2} }{8\sin \frac{x-y+2z}{2}}
 \\ &\quad + \frac{\cos \frac{(2n+3)(x+y)+4z}{2}  }{8\sin \frac{x+y-2z}{2}} -  \frac{\cos \frac{x+y + (8+4n)z}{2} }{8\sin \frac{x+y-2z}{2}} +  \frac{\cos \frac{x+y - (8+4n)z}{2} }{8\sin \frac{x+y+2z}{2}} - \frac{\cos \frac{(2n+3)(x+y)-4z}{2} }{8\sin \frac{x+y+2z}{2}}.
\end{align*}
\end{lemma}
\begin{proof}
We write $\Upsilon$ instead of the more cumbersome $\Upsilon_n(x,y,z)$. Since $\sin a  \sin b = \frac{1}{2} (\cos(a-b)-\cos(a+b))$, we have 
\begin{align*}
    \Upsilon=\frac{1}{2} \sum_{j=0}^{n}  (\cos(j+1) (x - y) -  \cos(j+1)(x + y)) \sin (2n-2j+3)z.
 \end{align*}
Now, using $\sin(a - b)=\sin a \cos b - \sin b \cos a$, we get
\begin{align*}
   2 \Upsilon =& \sin (2n+5)z \sum_{j=0}^{n} \cos(j+1) (x - y) \cos (j+1)2z 
             \\ &- \cos (2n+5)z \sum_{j=0}^{n} \cos(j+1) (x - y) \sin (j+1)2z
             \\ &-  \sin (2n+5)z \sum_{j=0}^{n} \cos(j+1) (x + y) \cos (j+1)2z
             \\ &+ \cos (2n+5)z \sum_{j=0}^{n} \cos(j+1) (x + y) \sin (j+1)2z.
 \end{align*}  
Let $x-y=\sigma$ and $x+y=\tau$. Then, using  (\ref{lem2}) and (\ref{lem3}) twice,  
\begin{align*}
    \Upsilon =& \sin (2n+5)z \left( \frac{ \sin\frac{(2n+3)(\sigma -2z)}{2} - \sin  \frac{\sigma-2z}{2}}{ 8 \sin  \frac{\sigma-2z}{2} } + \frac{ \sin\frac{(2n+3)(\sigma +2z)}{2} - \sin  \frac{\sigma +2z}{2} }{ 8 \sin  \frac{\sigma +2z}{2} } \right) 
              \\ &- \cos (2n+5)z \left( \frac{  \cos\frac{(2n+3)(\sigma -2z)   }{2} - \cos \frac{\sigma - 2z}{2}}{ 8 \sin  \frac{\sigma -2z}{2} } + \frac{ \cos \frac{\sigma +2z}{2} - \cos\frac{(2n+3)(\sigma +2z)}{2}}{  8 \sin  \frac{\sigma +2z}{2} }  \right)
            \\ &-  \sin (2n+5)z \left( \frac{ \sin\frac{(2n+3)(\tau -2z)}{2} - \sin  \frac{\tau-2z}{2}}{ 8 \sin  \frac{\tau-2z}{2} } + \frac{ \sin\frac{(2n+3)(\tau +2z)}{2}- \sin  \frac{\tau +2z}{2}}{ 8 \sin  \frac{\tau +2z}{2} }  \right)
        \\ &+ \cos (2n+5)z \left( \frac{  \cos\frac{(2n+3)(\tau -2z)   }{2} - \cos \frac{\tau - 2z}{2}}{ 8 \sin  \frac{\tau -2z}{2} } + \frac{ \cos \frac{\tau +2z}{2} - \cos\frac{(2n+3)(\tau +2z)}{2}}{  8 \sin  \frac{\tau +2z}{2} }  \right).
 \end{align*}
 By applying $\cos(a \pm b)= \cos a \cos b \mp \sin a \sin b$, we obtain the required result.    
\end{proof}

\section{Estimates of the Lebesgue function}

Before we formulate the main theorem of the paper, we show an estimate of the Lebesgue function on a sub-square of $Q$. 
By Theorem \ref{cub_formula}, 
for any $\theta, \phi \in [0,\pi],$ the Lebesgue function $\lambda_n(\cos\theta,\cos\phi)$ can be written as follows 
\begin{eqnarray*}
& & C_n \sum_{m=1}^{n+1} \sum_{k=1}^{\tfrac{n}2+1} \sin^2 w_m \, \sin^2 v_k \left| \sum_{j=0}^n \sum_{i=0}^{n-j} U_i(\cos v_k) U_i(\cos \theta) U_j(\cos(w_m) U_j(\cos\phi)\right|\\ 
&=& C_n \sum_{m=1}^{n+1} \sum_{k=1}^{\tfrac{n}2+1} \sin w_m \, \sin v_k \left| \sum_{j=0}^n \sum_{i=0}^{n-j} \frac{\sin (i+1) v_k \, \sin (i+1)\theta \, \sin(j+1)w_m \, \sin(j+1)\phi}{\sin \theta \, \sin \phi}\right| \\
&=& C_n \sum_{m=1}^{n+1} \sum_{k=1}^{\tfrac{n}2+1} \left| \sum_{j=0}^n  \frac{\sin w_m \, \sin(j+1)w_m \, \sin(j+1)\phi}{\sin \phi} \, \sum_{i=0}^{n-j}  \frac{\sin v_k \, \sin (i+1) v_k \, \sin (i+1)\theta}{\sin \theta }\right|. 
\end{eqnarray*}
For $\phi\in[\tfrac\pi6,\tfrac{5\pi}6]$, it is clear that 
\begin{align*} 
\sum_{j=0}^n  \left| \frac{\sin w_m \, \sin(j+1)w_m \, \sin(j+1)\phi}{\sin \phi} \right| &\le \sum_{j=0}^n  2 \sin w_m \, | \sin(j+1)w_m | \, |\sin(j+1)\phi| \\ &\le 2(n+1)
\end{align*}
and if we take also $\theta\in[\tfrac\pi6,\tfrac{5\pi}6]$ then 
\[ \lambda_n(\cos\theta,\cos\phi) \le C_n \sum_{m=1}^{n+1} \sum_{k=1}^{\tfrac{n}2+1}  4(n+1)^2 = 4 C_n (n+1)^3(\tfrac{n}2+1). \]
Consequently,
\[ \lambda_n(x,y) \le 4 C_n (n+1)^3(\tfrac{n}2+1) = {\mathcal{O}}(n^2) \]  
for any $(x,y)\in [-\tfrac{\sqrt{3}}2,\tfrac{\sqrt3}2]^2$. In the same manner, for any fixed $\delta\in (0,1)$ (independent of $n$), that is $(x,y)\in [-\delta,\delta]^2$, we can prove that 
\[ \lambda_n(x,y) \le {\mathcal{O}}(n^2)\,. \]  
 It is worth emphasizing that the obtained estimate can be improved. This, however, requires more tedious reasoning, as will be seen in the proof of the following theorem.
\vskip2mm

\begin{theorem}
    The Lebesgue constant of the  Morrow-Patterson points has the upper bound
\begin{equation} \label{lebMPnew}
\Lambda_n^{MP} \le \mathcal{O}(n^2)\;.
\end{equation}
\end{theorem}
\begin{proof}
By Theorem \ref{cub_formula}, for any $\theta, \phi \in [0,\pi]$, we can write
\[
  \lambda_n(\cos\theta,\cos\phi)= \sum_{m=1}^{n+1} \sum_{k=1}^{\tfrac{n}2+1} K_n((\theta,\phi);(w_m,v_k)),
\]
where $K_n((\theta,\phi);(w_m,v_k))$ is given by 
\begin{align*}
  K_n&=\frac{8\sin w_m \sin v_k }{(n+2)(n+3)} \left| \sum_{i=0}^{n} \sum_{j=0}^{ n-i }  \frac{  \sin(i+1) w_m \sin(i+1) \theta \sin(j+1) v_k \sin(j+1) \phi  }{ \sin \phi \sin \theta   }\right| 
  \\ &= \frac{8\sin w_m \sin v_k }{(n+2)(n+3)} \left| \sum_{0\le i+j \le n}  \frac{  \sin(i+1) w_m \sin(i+1) \theta \sin(j+1) v_k \sin(j+1) \phi  }{ \sin \phi \sin \theta   }\right| 
  \\ &= \frac{8\sin w_m \sin v_k }{(n+2)(n+3)} \left| \sum_{j=0}^{n} \sum_{i=0}^{ n-j }  \frac{  \sin(i+1) w_m \sin(i+1) \theta \sin(j+1) v_k \sin(j+1) \phi  }{ \sin \phi \sin \theta   }\right|.
\end{align*}
In order to prove that $
  \lambda_n(\cos\theta,\cos\phi) \le \mathcal{O}(n^2) 
  $ for any $\theta, \phi \in [0,\pi],\,$
according to (\ref{Leb_const_mesh}), it is enough to verify that $
  \lambda_n(\cos\theta,\cos\phi)\le \mathcal{O}(n^2) 
  $ for any $\theta, \phi \in [\frac{c}{n},\pi-\frac{c}{n}]$, where $c>0$ is a fixed constant independent of $n$. By Lemma \ref{symLf}, $\lambda_n(\cos\theta,\cos\phi)=\lambda_n(\cos(\pi - \theta),\cos\phi)$. Therefore, we can assume, without 
loss of generality, that $\theta \in [\frac{c}{n},\frac{\pi}{2}]$. Moreover, for any fixed $\delta\in (0,\frac{\pi}{2})$ (independent of $n$), we have already observed that 
\[ \lambda_n(\cos\theta,\cos\phi) \leq {\mathcal{O}}(n^2) \]  for  $(\theta,\phi)\in [\delta,\pi - \delta]^2$. To show the result for the remaining $\theta$ and $\phi$, let us first consider the case when $\theta=w_m$ for some $m$. Then
\begin{align*}
  K_n ((\theta,\phi);(\theta, v_k))= \frac{8\sin v_k  }{(n+2)(n+3)} \left| \sum_{j=0}^{n}   \frac{  \sin(j+1) v_k \sin(j+1) \phi   }{ \sin \phi   }  \sum_{i=0}^{n-j} \sin^2(i+1)\theta \right|.
\end{align*}
Since $\left|\tfrac{\sin(j+1)\phi}{\sin \phi}\right|\le j+1$ and $\sum\limits_{i=0}^{l} \sin^2(i+1)\theta = \frac{1}{4} (3+2l - \frac{\sin(3+2l)\theta }{\sin \theta})$, we have
\begin{align*}
  K_n ((\theta,\phi);(\theta,v_k)) \leq& \frac{4 }{(n+2)(n+3) } \sum_{j=0}^{n}   (j+1)(2(n-j)+3) 
  \\ &=\frac{4 }{6(n+2)(n+3) } (n+1)(n+2)(2n+9)=\mathcal{O}(n).
\end{align*}
Therefore, for such $m$, we have 
\begin{align}\label{1}
\sum_{k=1}^{\frac{n}{2}+1} K_n ((\theta,\phi);(\theta,v_k)) \le  \mathcal{O}(n^2).
\end{align} 
Suppose now that $w_m\neq  \theta$, then by Lemma \ref{lem4},
\begin{align*}
  K_n&((\theta,\phi);(w_m,v_k))\\ = &\frac{8\sin v_k  }{(n+2)(n+3)}  \left| \sum_{j=0}^{n}   \frac{  \sin(j+1) v_k \sin(j+1) \phi   }{ \sin \phi   }  \left\{ \frac{ \cos \frac{w_m}{2} \cos \frac{\theta}{2} }{ 2\sin  \theta} \left( S_1-S_2 \right) \right. \right. \\ &\quad \quad \quad \quad \quad \quad\quad \quad\quad \quad\quad \quad\quad \quad + \left. \left. \frac{\cos \frac{w_m}{2}      }{ 4\cos \frac{\theta}{2}  } \left(  \frac{\cos \frac{w_m - \theta}{2}}{ \sin \frac{w_m - \theta}{2}}   S_1
   + \frac{\cos \frac{w_m + \theta}{2}}{ \sin \frac{w_m + \theta}{2}} S_2 \right) \right\} \right|,
\end{align*}
where $S_1=\sin\frac{(2(n-j)+3)(w_m - \theta)}{2}$, $S_2=\sin\frac{(2(n-j)+3)(w_m + \theta)}{2}$. 
Hence,
\begin{align*}
  K_n&((\theta,\phi);(w_m,v_k))\leq \frac{8  }{(n+2)(n+3)}   \left(\sum_{j=0}^{n}   \frac{  j+1  }{  \sin  \theta  }    + \sum_{j=0}^{n} \frac{  (j+1)(2(n-j)+3)   }{ 2\cos \frac{\theta}{2}  }  \right).
\end{align*}
Since $\theta \in [\frac{c}{n},\frac{\pi}{2}]$ and using the fact that $\sin x \geq \frac{2}{\pi}x$ for $x \in [0, \frac{\pi}{2}]$ we have
\begin{align}\label{pom1}
  K_n&((\theta,\phi);(w_m,v_k))\leq \frac{8(n+1)(n+2)  }{(n+2)(n+3)}   \left(\frac{\pi n }{4c}       + \frac{  2n+9   }{ 6\sqrt{2}  }  \right) = \mathcal{O}(n).
\end{align}
If we assume, in addition, that  $\phi \in [\delta, \pi - \delta]$, then
\begin{align*}
  K_n&((\theta,\phi);(w_m,v_k))\leq \frac{8(n+1)  }{ (n+2)(n+3)\sin \delta}   \left(\frac{\pi n }{2c}       + \frac{  n+3   }{ \sqrt{2}  }  \right).
\end{align*}
For fixed $\theta,\phi$ and $m$, let 
\begin{align*}
     &V_1^m:= \left\{
                \begin{array}{ll}
                  \{1 \leq k \leq  \frac{n}{2} + 1 : |\frac{v_k - \phi - w_m + \theta}{2}| < \frac{\pi}{n+3}\} , & \hbox{if } \frac{\phi + w_m - \theta}{2} \in [-\frac{\pi}{4}, \frac{3\pi}{4}], \\
                  \{1 \leq k \leq  \frac{n}{2} + 1 : |\frac{v_k - \phi - w_m + \theta}{2} + \pi| < \frac{\pi}{n+3}\} , & \hbox{if } \frac{ \phi + w_m - \theta}{2} \in (\frac{3\pi}{4},\pi].
                \end{array}
              \right.
      \\
     &V_2^m:= \left\{
                \begin{array}{ll}
                \{1 \leq k \leq  \frac{n}{2} + 1 : |\frac{v_k - \phi + w_m - \theta}{2} - \pi| < \frac{\pi}{n+3}\} , & \hbox{if } \frac{\phi - w_m + \theta}{2} \in [-\frac{\pi}{2}, -\frac{\pi}{4}], \\
                  \{1 \leq k \leq  \frac{n}{2} + 1 : |\frac{v_k - \phi + w_m - \theta}{2}| < \frac{\pi}{n+3}\} , & \hbox{if } \frac{\phi - w_m + \theta}{2} \in (-\frac{\pi}{4}, \frac{3\pi}{4}].
                \end{array}
              \right.
      \\
     &V_3^m:= \left\{
                \begin{array}{ll}
                  \{1 \leq k \leq  \frac{n}{2} + 1 : |\frac{v_k + \phi - w_m + \theta}{2}-\pi| < \frac{\pi}{n+3}\} , & \hbox{if } \frac{-\phi + w_m - \theta}{2} \in [-\frac{3\pi}{4}, -\frac{\pi}{4}], \\
                  \{1 \leq k \leq  \frac{n}{2} + 1 : |\frac{v_k + \phi - w_m + \theta}{2}| < \frac{\pi}{n+3}\} , & \hbox{if } \frac{ -\phi + w_m - \theta}{2} \in (-\frac{\pi}{4},\frac{\pi}{2}].
                \end{array}
              \right.
      \\
     &V_4^m:= \left\{
                \begin{array}{ll}
                  \{1 \leq k \leq  \frac{n}{2} + 1 : |\frac{v_k + \phi + w_m - \theta}{2}-\pi| < \frac{\pi}{n+3}\} , & \hbox{if } \frac{-\phi - w_m + \theta}{2} \in [-\pi, -\frac{\pi}{4}], \\
                  \{1 \leq k \leq  \frac{n}{2} + 1 : |\frac{v_k + \phi + w_m - \theta}{2} | < \frac{\pi}{n+3}\} , & \hbox{if } \frac{ -\phi - w_m + \theta}{2} \in (-\frac{\pi}{4},\frac{\pi}{4}].
                \end{array}
              \right.
      \\
     &V_5^m:= \left\{
                \begin{array}{ll}
                  \{1 \leq k \leq  \frac{n}{2} + 1 : |\frac{v_k - \phi - w_m - \theta}{2}| < \frac{\pi}{n+3}\} , & \hbox{if } \frac{\phi + w_m + \theta}{2} \in [0, \frac{3\pi}{4}], \\
                  \{1 \leq k \leq  \frac{n}{2} + 1 : |\frac{v_k - \phi - w_m - \theta}{2} + \pi| < \frac{\pi}{n+3}\} , & \hbox{if } \frac{ \phi + w_m + \theta}{2} \in (\frac{3\pi}{4},\frac{5\pi}{4}].
                \end{array}
              \right.
      \\
     &V_6^m:= \left\{
                \begin{array}{ll}
                \{1 \leq k \leq  \frac{n}{2} + 1 : |\frac{v_k - \phi + w_m + \theta}{2} - \pi| < \frac{\pi}{n+3}\} , & \hbox{if } \frac{ \phi - w_m - \theta}{2} \in [-\frac{3\pi}{4},-\frac{\pi}{4}], \\
                  \{1 \leq k \leq  \frac{n}{2} + 1 : |\frac{v_k - \phi + w_m + \theta}{2}| < \frac{\pi}{n+3}\} , & \hbox{if } \frac{\phi - w_m - \theta}{2} \in (-\frac{\pi}{4}, \frac{\pi}{2}].
                \end{array}
              \right.
      \\
     &V_7^m:= \left\{
                \begin{array}{ll}
                 \{1 \leq k \leq  \frac{n}{2} + 1 : |\frac{v_k + \phi - w_m - \theta}{2} - \pi| < \frac{\pi}{n+3}\} , & \hbox{if } \frac{ -\phi + w_m + \theta}{2} \in [-\frac{\pi}{2},-\frac{\pi}{4}],
\\
                  \{1 \leq k \leq  \frac{n}{2} + 1 : |\frac{v_k + \phi - w_m - \theta}{2}| < \frac{\pi}{n+3}\} , & \hbox{if } \frac{-\phi + w_m + \theta}{2} \in (-\frac{\pi}{4}, \frac{3\pi}{4}].
                \end{array}
              \right.
      \\
     &V_8^m:= \left\{
                \begin{array}{ll}
                \{1 \leq k \leq  \frac{n}{2} + 1 : |\frac{v_k + \phi + w_m + \theta}{2} - \pi| < \frac{\pi}{n+3}\} , & \hbox{if } \frac{ -\phi - w_m - \theta}{2} \in [-\frac{5\pi}{4},-\frac{\pi}{4}],
\\
                  \{1 \leq k \leq  \frac{n}{2} + 1 : |\frac{v_k + \phi + w_m + \theta}{2}| < \frac{\pi}{n+3}\} , & \hbox{if } \frac{-\phi - w_m - \theta}{2} \in (-\frac{\pi}{4}, 0].
                \end{array}
              \right.
\end{align*}
Since $\frac{v_{k+1} -v_k}{2}=\frac{\pi}{n+3}$, each $V^m_i$ is a set of at most two elements (these sets may be empty). Let $\Omega:=V_1^m \cup \ldots \cup V_8^m$.  If $\Omega$ is nonempty and $w_{m} \neq \theta$ for $m=1,\ldots,n+1$, then by (\ref{pom1})
 \begin{align}\label{pom2}
 \sum_{m=1}^{n+1} \sum_{l\in \Omega} K_n&((\theta,\phi);(w_m,v_l)) \le \mathcal{O}(n^2).
\end{align}
Now, for fixed $\theta$, let $\Theta:= \{ 1 \leq m \leq  n + 1 : \, | w_m - \theta|  < \frac{\pi}{n+2}  \}$. 
If $\Theta \neq \emptyset$, then by (\ref{1}), (\ref{pom1}) and the fact that $\Theta$ is a set of at most two elements we have
\begin{align}\label{pom22}
\sum_{m \in \Theta} \sum_{k=1}^{\frac{n}{2}+1}  K_n ((\theta,\phi);(w_m, v_k))  \le  \mathcal{O}(n^2).
\end{align} 
For the remaining indexes $m$ and $k$, $m \notin \Theta$, $k \notin \Omega$, by using Lemma \ref{lem5} 
\begin{align*}
  K_n&((\theta,\phi);(w_m,v_k))\\ = &\frac{8\sin v_k  }{(n+2)(n+3)}  \left| \ \frac{ \cos \frac{w_m}{2} \cos \frac{\theta}{2} }{ 2\sin \phi \sin  \theta} \left[ \Upsilon_n\left(v_k,\phi, \frac{w_m - \theta}{2}\right)-\Upsilon_n\left(v_k,\phi, \frac{w_m + \theta}{2}\right) \right] \right. \\ & \quad   +  \left. \frac{\cos \frac{w_m}{2}      }{ 4\sin \phi\cos \frac{\theta}{2}  } \left[  \frac{\cos \frac{w_m - \theta}{2}}{ \sin \frac{w_m - \theta}{2}}   \Upsilon_n\left(v_k,\phi, \frac{w_m - \theta}{2}\right)
   + \frac{\cos \frac{w_m + \theta}{2}}{ \sin \frac{w_m + \theta}{2}} \Upsilon_n\left(v_k,\phi, \frac{w_m + \theta}{2}\right) \right]  \right|.
\end{align*}
Therefore,
 \begin{align*}
  K_n&((\theta,\phi);(w_m,v_k))
  \\ \leq &\frac{8\sin v_k  }{(n+2)(n+3)}  \left| \ \frac{ \cos \frac{w_m}{2} \cos \frac{\theta}{2} }{ 2\sin \phi \sin  \theta} \left[ \Upsilon_n\left(v_k,\phi, \frac{w_m - \theta}{2}\right)-\Upsilon_n\left(v_k,\phi, \frac{w_m + \theta}{2}\right) \right] \right| 
  \\&+ \frac{8\sin v_k  }{(n+2)(n+3)}   \left| \frac{\cos \frac{w_m}{2}      }{ 4\sin \phi\cos \frac{\theta}{2}  }   \frac{\cos \frac{w_m - \theta}{2}}{ \sin \frac{w_m - \theta}{2}}   \Upsilon_n\left(v_k,\phi, \frac{w_m - \theta}{2}\right) \right| 
  \\&+ \frac{8\sin v_k  }{(n+2)(n+3)} 
   \left| \frac{\cos \frac{w_m}{2}      }{ 4\sin \phi\cos \frac{\theta}{2}  }  \frac{\cos \frac{w_m + \theta}{2}}{ \sin \frac{w_m + \theta}{2}} \Upsilon_n\left(v_k,\phi, \frac{w_m + \theta}{2}\right)   \right|.
\end{align*}
To show $ \lambda_n(\cos\theta,\cos\phi)  \le {\mathcal{O}}(n^2)$, we must consider each of the three terms of the above sum separately. Let's start with the last one. According to Lemma \ref{lem5}, the value of $\Upsilon_n(v_k,\phi, \frac{w_m + \theta}{2})$ could be expressed as the sum of eight terms. Let us consider the first (in accordance with the order given in Lemma \ref{lem5}) of them, namely
\begin{align*}
   I_{3,1}:=  \frac{8\sin v_k  }{(n+2)(n+3)} 
   \left| \frac{\cos \frac{w_m}{2}      }{ 4\sin \phi\cos \frac{\theta}{2}  }  \frac{\cos \frac{w_m + \theta}{2}}{ \sin \frac{w_m + \theta}{2}} \frac{\cos \frac{v_k-\phi + (4+2n)(w_m + \theta)}{2} }{8\sin \frac{v_k-\phi-w_m - \theta}{2}}  \right|.
\end{align*}
As we noted earlier, without loss of generality, we can assume that $\theta \in  [\frac{c}{n}, \frac{\pi}{2}]$. Hence,
\begin{align}\label{I31}
   I_{3,1}\leq  \frac{2\sqrt{2}  }{(n+2)(n+3) \sin \phi} \cdot
   \frac{1 }{|\sin \frac{w_m + \theta}{2} \sin \frac{v_k-\phi-w_m - \theta}{2}|}  .
\end{align}
It is clear that
\begin{align}\label{sinsq}
   \frac{1 }{|\sin \frac{w_m + \theta}{2} \sin \frac{v_k-\phi-w_m - \theta}{2}|}  \leq  \frac{1}{\sin^2 \frac{v_k-\phi-w_m - \theta}{2}}  + \frac{1}{\sin^2 \frac{w_m + \theta}{2}}.
\end{align}
It is known that a necessary (and sufficient) condition for a continuous function $f$ to be convex on an open interval $I$ is that (see, exercise 1, on page 63 in \cite{NP})
\begin{align}\label{fconvex}
f(x) \leq \frac{1}{2h} \int_{x-h}^{x+h} f(t) \, dt
\end{align}
for all $x$ and $h$ with $[x - h, x + h] \subset I$. Let $\sigma=\frac{\phi+w_m + \theta}{2}$ and consider the case when $\sigma\in [0, \frac{3 \pi}{4}]$. Then, for $1\leq k \leq \frac{n}{2}+1$, $- \frac{3 \pi}{4} \leq \frac{v_k}{2} - \sigma \leq \frac{\pi}{2}$. Define 
\begin{align*}
  &L:=\{ 1\leq k \leq \frac{n}{2}+1 : \frac{v_k}{2} - \sigma < 0, k \notin V^m_5 \}, \\
  &R:=\{ 1\leq k \leq \frac{n}{2}+1 : \frac{v_k}{2} - \sigma \geq 0, k \notin V^m_5 \}.
\end{align*}
Let us assume that $L$ and $R$ are nonempty sets (if one of them is empty, the part related to it should be omitted in the following reasoning). Let $l_{max}$ be the greatest element of $L$ and $r_{min}$ least element of $R$.
 Then, by (\ref{fconvex}), we have 
\begin{align*}
 &\sum\limits_{k \in L} \frac{1}{\sin^2 (\frac{v_k}{2} - \sigma)}  \leq \frac{n+3}{\pi} \sum\limits_{k=1}^{l_{max}} \int_{\frac{v_k}{2}  - \frac{\pi}{2(n+3)}}^{\frac{v_k}{2}  + \frac{\pi}{2(n+3)}}  \frac{1}{ \sin^2 (t - \sigma)} \, dt,
 \\& \sum\limits_{k \in R} \frac{1}{\sin^2 (\frac{v_k}{2} - \sigma)}  \leq \frac{n+3}{\pi} \sum\limits_{k=r_{min}}^{ \frac{n}{2}+1} \int_{\frac{v_k}{2}  - \frac{\pi}{2(n+3)}}^{\frac{v_k}{2}  + \frac{\pi}{2(n+3)}}  \frac{1}{ \sin^2 (t - \sigma)} \, dt.
\end{align*}
Since $\frac{v_{k+1}}{2} - \frac{v_{k}}{2}=\frac{\pi}{n+3}$, it follows that 
\begin{align*}
  &\sum\limits_{k=1}^{l_{max}} \int_{\frac{v_k}{2}  - \frac{\pi}{2(n+3)}}^{\frac{v_k}{2}  + \frac{\pi}{2(n+3)}}  \frac{1}{ \sin^2 (t - \sigma)} \, dt=  \int_{\frac{v_1}{2}  - \frac{\pi}{2(n+3)}}^{\frac{v_{l_{max}}}{2}  + \frac{\pi}{2(n+3)}}  \frac{1}{ \sin^2 (t - \sigma)} \, dt,
  \\&\sum\limits_{k=r_{min}}^{ \frac{n}{2}+1} \int_{\frac{v_k}{2}  - \frac{\pi}{2(n+3)}}^{\frac{v_k}{2}  + \frac{\pi}{2(n+3)}}  \frac{1}{ \sin^2 (t - \sigma)} \, dt = \int_{\frac{v_{r_{min}}}{2}  - \frac{\pi}{2(n+3)}}^{\frac{v_{\frac{n}{2}+1}}{2}  + \frac{\pi}{2(n+3)}}  \frac{1}{ \sin^2 (t - \sigma)} \, dt .
\end{align*}
Thus, by the definitions of $L$ and $R$, we have
 \begin{align*}
 &\sum\limits_{k \in L} \frac{1}{\sin^2 (\frac{v_k}{2} - \sigma)}  \leq \frac{n+3}{\pi} \int_{ \frac{\pi}{2(n+3)}}^{\pi -\frac{\pi}{2(n+3)}}  \frac{1}{ \sin^2 t} \, dt,
 \\& \sum\limits_{k \in R} \frac{1}{\sin^2 (\frac{v_k}{2} - \sigma)}  \leq \frac{n+3}{\pi}  \int_{ \frac{\pi}{2(n+3)}}^{\pi -\frac{\pi}{2(n+3)}}  \frac{1}{ \sin^2 t} \, dt.
\end{align*}
Using Jordan's inequality, we have
\begin{align*}
  \int_{ \frac{\pi}{2(n+3)}}^{\pi - \frac{\pi}{2(n+3)}}  \frac{1}{ \sin^2 t} \, dt = 2 \cot\left(\frac{\pi}{2(n+3)}\right) \leq 2(n+3).
\end{align*}
Therefore,
\begin{align}\label{part1a}
 \sum\limits_{m=1}^{n+1}  \sum\limits_{k\in L \cup R} \frac{1}{\sin^2  \frac{v_k-\phi-w_m - \theta}{2} } \leq \frac{4(n+1)(n+3)^2}{\pi}.
\end{align}
It is noteworthy that $L \cup R \cup V^m_5 = \{1,2,\ldots, \frac{n}{2}+1\}$. 
If $\sigma \in (\frac{3 \pi}{4}, \frac{5 \pi}{4}]$, then $\sigma - \pi \in (-\frac{ \pi}{4},\frac{ \pi}{4}]$ and $- \frac{ \pi}{4} \leq \frac{v_k}{2} - \sigma + \pi \leq \frac{3\pi}{4}$. Since  $\sin^2(\frac{v_k}{2} - \sigma) = \sin^2(\frac{v_k}{2} - \sigma + \pi)$, a similar reasoning as in the case $\sigma\in [0, \frac{3 \pi}{4}]$ yields 
\begin{align}\label{part1b}
 \sum\limits_{m=1}^{n+1}  \sum\limits_{k\in L' \cup R'} \frac{1}{\sin^2  \frac{v_k-\phi-w_m - \theta}{2} } \leq \frac{4(n+1)(n+3)^2}{\pi},
\end{align}
where
\begin{align*}
  &L':=\{ 1\leq k \leq \frac{n}{2}+1 : \frac{v_k}{2} - \sigma + \pi < 0, k \notin V^m_5 \}, \\
  &R':=\{ 1\leq k \leq \frac{n}{2}+1 : \frac{v_k}{2} - \sigma + \pi \geq 0, k \notin V^m_5 \}.
\end{align*}
As before, we have $L' \cup R' \cup V^m_5 = \{1,2,\ldots, \frac{n}{2}+1\}$.
Next we will show that
\begin{align}\label{part2}
 \sum_{m=1}^{n+1} \sum_{k=1}^{\tfrac{n}2+1}  \frac{1}{\sin^2  \frac{w_m + \theta}{2}} \le{\mathcal{O}}(n^3).
\end{align}
Since $\frac{\pi}{2(n+2)} \leq \frac{w_1 + \theta}{2}$, $\frac{3\pi}{4} \geq \frac{w_{n+1} + \theta}{2} $ and
   $\frac{w_s}{2} - \frac{w_{s-1}}{2} =  \frac{ \pi}{2(n+2)}$, for      $s=2,\ldots,   n+1$,             proceeding as before, one can show that
 \begin{align*}
 \sum_{m=1}^{n+1} \frac{1}{\sin^2  \frac{w_m + \theta}{2}} \leq \frac{2(n+2)}{\pi} \int_{ \frac{\pi}{4(n+2)}}^{\pi -\frac{\pi}{4(n+2)}}  \frac{1}{ \sin^2 t} \, dt \leq \frac{8(n+2)^2}{\pi},
\end{align*}
and therefore (\ref{part2}). 
Combining inequalities (\ref{I31}), (\ref{sinsq}), (\ref{part1a}) or (\ref{part1b}) (depending on the case), (\ref{part2}) and the fact that $\phi \in  [\frac{c}{n},\pi - \frac{c}{n}]$ leads to 
\begin{align*}
  \sum_{m=1}^{n+1}  \sum_{k\in K^m_5} \frac{8\sin v_k  }{(n+2)(n+3)} 
   \left| \frac{\cos \frac{w_m}{2}      }{ 4\sin \phi\cos \frac{\theta}{2}  }  \frac{\cos \frac{w_m + \theta}{2}}{ \sin \frac{w_m + \theta}{2}} \frac{\cos \frac{v_{k}-\phi + (4+2n)(w_{m} + \theta)}{2} }{8\sin \frac{v_{k}-\phi-w_{m} - \theta}{2}}  \right| \le  \mathcal{O}(n^2),
\end{align*}
where $K^m_5:=\{1,2,\ldots, \frac{n}{2}+1\} \setminus V^m_5$. Proceeding similarly for the remaining components of 
$
\frac{8\sin v_k  }{(n+2)(n+3)}   \cdot \frac{\cos \frac{w_m}{2}      }{ 4\sin \phi\cos \frac{\theta}{2}  }   \frac{\cos \frac{w_m + \theta}{2}}{ \sin \frac{w_m + \theta}{2}}   \Upsilon_n\left(v_k,\phi, \frac{w_m + \theta}{2}\right) $, we get
\begin{align}\label{I3}
    \sum_{m=1}^{n+1}  \sum_{k\in K} \frac{8\sin v_k  }{(n+2)(n+3)} 
   \left| \frac{\cos \frac{w_m}{2}      }{ 4\sin \phi\cos \frac{\theta}{2}  }  \frac{\cos \frac{w_m + \theta}{2}}{ \sin \frac{w_m + \theta}{2}} \Upsilon_n\left(v_k,\phi, \frac{w_m + \theta}{2}\right)   \right| \le \mathcal{O}(n^2).
\end{align}
Here $K:=\{k \in \mathbb{N}: 1 \leq k \leq \frac{n}{2}+1, \, k \notin \Omega\}$.
Now consider the associative term of $ \frac{\cos \frac{w_m}{2}      }{ 4\sin \phi\cos \frac{\theta}{2}  }   \frac{\cos \frac{w_m - \theta}{2}}{ \sin \frac{w_m - \theta}{2}}   \Upsilon_n(v_k,\phi, \frac{w_m - \theta}{2}) $ (which we shall emphasize with the number 2 - one of the two subscripts of the letter $I$). Let
\begin{align*}
   I_{2,5}:=  \frac{8\sin v_k  }{(n+2)(n+3)} 
   \left| \frac{\cos \frac{w_m}{2}      }{ 4\sin \phi\cos \frac{\theta}{2}  }  \frac{\cos \frac{w_m - \theta}{2}}{ \sin \frac{w_m - \theta}{2}} \frac{\cos \frac{(2n+3)(v_k+\phi) + 2(w_m - \theta)}{2} }{8\sin \frac{v_k+\phi-w_m + \theta}{2}}  \right|.
\end{align*}
Here the subscript $5$ denotes the fifth term of expression $\Upsilon_n(v_k,\phi, \frac{w_m - \theta}{2})$  according to Lemma \ref{lem5}.
Since $\theta \in [\frac{c}{n}, \frac{\pi}{2}]$, we have
\begin{align*}
   I_{2,5} \leq  \frac{2\sqrt{2} }{(n+2)(n+3) \sin \phi} \cdot
   \frac{1 }{|\sin \frac{w_m - \theta}{2} \sin \frac{v_k+\phi-w_m + \theta}{2}|}  .
\end{align*}
Let $K^m_3:=\{1,2,\ldots, \frac{n}{2}+1\} \setminus V^m_3$ and let $M:=\{1,2,\ldots,n+1\} \setminus \Theta$. A similar reasoning as for $\frac{v_k-\phi-w_m - \theta}{2}$ gives
\begin{align*}
   &\sum_{k \in K^m_3} \frac{1}{\sin^2  \frac{v_{k}+\phi-w_{m} + \theta}{2}} \leq \frac{4(n+3)^2}{\pi},
\\ &\sum_{m \in M}  \frac{1}{\sin^2  \frac{w_m - \theta}{2}} \leq \frac{16(n+2)^2}{\pi}.
\end{align*}
Hence, 
\begin{align*}
 &\sum_{m=1}^{n+1}  \sum_{k \in K^m_3} \frac{1}{\sin^2  \frac{v_{k}+\phi-w_{m} + \theta}{2}} \leq \frac{4(n+1)(n+3)^2}{\pi}, \\
 &\sum_{m \in M}  \sum_{k=1}^{\tfrac{n}2+1} \frac{1}{\sin^2  \frac{w_{m} - \theta}{2}} \leq \frac{16(\tfrac{n}2+1)(n+2)^2}{\pi}.
\end{align*}
Thus if  $\phi \in  [\frac{c}{n},\pi - \frac{c}{n}]$, then
\begin{align*}
 \sum_{m \in M}  \sum_{k \in K^m_3}   \frac{8\sin v_{k}  }{(n+2)(n+3)} 
   \left| \frac{\cos \frac{w_{m}}{2}      }{ 4\sin \phi\cos \frac{\theta}{2}  }  \frac{\cos \frac{w_{m} - \theta}{2}}{ \sin \frac{w_{m} - \theta}{2}} \frac{\cos \frac{(2n+3)(v_{k}+\phi) + 2(w_{m} - \theta)}{2} }{8\sin \frac{v_{k}+\phi-w_{m} + \theta}{2}}  \right|\le {\mathcal{O}}(n^2).
\end{align*}
A similar proof can be carried out for the remaining components of the sum
\[\frac{8\sin v_k  }{(n+2)(n+3)}  \cdot \frac{\cos \frac{w_m}{2}      }{ 4\sin \phi\cos \frac{\theta}{2}  }   \frac{\cos \frac{w_m - \theta}{2}}{ \sin \frac{w_m - \theta}{2}}   \Upsilon_n(v_k,\phi, \frac{w_m - \theta}{2}),\]
i.e., $I_{2,1}, I_{2,2}, I_{2,3}, I_{2,4}$ and $I_{2,6}, I_{2,7}, I_{2,8}$, following our notation.
Therefore, 
\begin{align}\label{I2}
    \sum_{m \in M}  \sum_{k \in K}   \frac{8\sin v_{k}  }{(n+2)(n+3)} 
   \left| \frac{\cos \frac{w_{m}}{2}      }{ 4\sin \phi\cos \frac{\theta}{2}  }  \frac{\cos \frac{w_{m} - \theta}{2}}{ \sin \frac{w_{m} - \theta}{2}} \Upsilon_n(v_k,\phi, \frac{w_m - \theta}{2})  \right|\le {\mathcal{O}}(n^2).
\end{align}
It remains to consider
\begin{align*}
  I_1:=\frac{8\sin v_k  }{(n+2)(n+3)}  \left| \ \frac{ \cos \frac{w_m}{2} \cos \frac{\theta}{2} }{ 2\sin \phi \sin  \theta} \left( \Upsilon_n(v_k,\phi, \frac{w_m - \theta}{2})-\Upsilon_n(v_k,\phi, \frac{w_m + \theta}{2}) \right) \right|.
\end{align*}
We first consider the following expressions
\begin{align*}
  &A_1:=\left| \ \frac{ \cos \frac{w_m}{2} \cos \frac{\theta}{2} }{ 2\sin \phi \sin  \theta} \left( \frac{\cos \frac{v_{k}-\phi + (2n+4)(w_{m} - \theta)}{2} }{8\sin \frac{v_{k}-\phi-w_{m} + \theta}{2}} -   \frac{\cos \frac{v_{k}+\phi + (2n+4)(w_{m} - \theta)}{2} }{8\sin \frac{v_{k}+\phi-w_{m} + \theta}{2}} \right) \right|, 
  \\ 
  &A_2:=\left| \ \frac{ \cos \frac{w_m}{2} \cos \frac{\theta}{2} }{ 2\sin \phi \sin  \theta} \left( \frac{\cos \frac{v_{k}+\phi - (2n+4)(w_{m} - \theta)}{2} }{8\sin \frac{v_{k}+\phi+w_{m} - \theta}{2}} -  \frac{\cos \frac{v_{k}-\phi - (2n+4)(w_{m} - \theta)}{2} }{8\sin \frac{v_{k}-\phi+w_{m} - \theta}{2}} \right) \right|, 
  \\
  &A_3:=\left| \ \frac{ \cos \frac{w_m}{2} \cos \frac{\theta}{2} }{ 2\sin \phi \sin  \theta} \left( \frac{\cos \frac{v_{k}-\phi + (2n+4)(w_{m} + \theta)}{2} }{8\sin \frac{v_{k}-\phi-w_{m} - \theta}{2}} -   \frac{\cos \frac{v_{k}+\phi + (2n+4)(w_{m} + \theta)}{2} }{8\sin \frac{v_{k}+\phi- w_{m} - \theta}{2}} \right) \right|, 
  \\
  &A_4:=\left| \ \frac{ \cos \frac{w_m}{2} \cos \frac{\theta}{2} }{ 2\sin \phi \sin  \theta} \left( \frac{\cos \frac{v_{k}+\phi - (2n+4)(w_{m} + \theta)}{2} }{8\sin \frac{v_{k}+\phi+w_{m} + \theta}{2}} -  \frac{\cos \frac{v_{k}-\phi - (2n+4)(w_{m} + \theta)}{2} }{8\sin \frac{v_{k}-\phi+w_{m} + \theta}{2}} \right) \right|. 
\end{align*}
Here the first two are from $\Upsilon_n(v_k,\phi, \frac{w_m - \theta}{2})$, and the next two from $\Upsilon_n(v_k,\phi, \frac{w_m + \theta}{2})$. Let $z=(2n+4)(w_{m} - \theta)$, then by $\sin(a \pm b)=\sin a \cos b \pm \sin b \cos a$, we have
\begin{align*}
  A_1 &\quad \leq \left|     \frac{ \cos \frac{w_m}{2} \cos \frac{\theta}{2} \sin \frac{v_{k}-w_{m} + \theta}{2} \cos \frac{\phi}{2}  \left(\cos \frac{v_{k}-\phi + z}{2} - \cos \frac{v_{k}+\phi + z}{2}\right) }{ 16\sin \phi \sin  \theta \sin \frac{v_{k}-\phi-w_{m} + \theta}{2} \sin \frac{v_{k}+\phi-w_{m} + \theta}{2}}   \right|
   \\ &\quad \quad \quad + \left|     \frac{ \cos \frac{w_m}{2} \cos \frac{\theta}{2} \cos \frac{v_{k}-w_{m} + \theta}{2} \sin \frac{\phi}{2}  \left(\cos \frac{v_{k}-\phi + z}{2} + \cos \frac{v_{k}+\phi + z}{2}\right) }{ 16\sin \phi \sin  \theta \sin \frac{v_{k}-\phi-w_{m} + \theta}{2} \sin \frac{v_{k}+\phi-w_{m} + \theta}{2}}   \right|.
\end{align*}
Now by $\cos a + \cos b = 2 \cos \frac{a+b}{2} \cos \frac{a-b}{2}$ and $\cos a - \cos b = -2 \sin \frac{a+b}{2} \sin \frac{a-b}{2}$, 
\begin{align*}
  A_1 &\leq \left|     \frac{ \cos \frac{w_m}{2} \cos \frac{\theta}{2} \sin \frac{v_{k}-w_{m} + \theta}{2} \cos \frac{\phi}{2}  \sin \frac{v_{k} + z}{2}  \sin \frac{\phi}{2} }{ 8\sin \phi \sin  \theta \sin \frac{v_{k}-\phi-w_{m} + \theta}{2} \sin \frac{v_{k}+\phi-w_{m} + \theta}{2}}   \right|
   \\ &\quad \quad \quad + \left|     \frac{ \cos \frac{w_m}{2} \cos \frac{\theta}{2} \cos \frac{v_{k}-w_{m} + \theta}{2} \sin \frac{\phi}{2} \cos \frac{v_{k} + z}{2}  \cos \frac{\phi}{2}  }{ 8\sin \phi \sin  \theta \sin \frac{v_{k}-\phi-w_{m} + \theta}{2} \sin \frac{v_{k}+\phi-w_{m} + \theta}{2}}   \right|.
\end{align*}
Hence
\begin{align}\label{part31}
  A_1 \leq     \frac{ 1 }{ 8 \sin  \theta  \sin^2 \frac{v_{k}-\phi-w_{m} + \theta}{2}} + \frac{1}{8  \sin  \theta  \sin^2 \frac{v_{k}+\phi-w_{m} + \theta}{2}}.
\end{align}
We already know that
\begin{align}\label{part32}
\sum_{m=1}^{n+1}  \sum_{k \in K^m_3} \frac{1}{\sin^2  \frac{v_{k}+\phi-w_{m} + \theta}{2}} \leq \frac{4(n+1)(n+3)^2}{\pi}.
\end{align}
Similarly, we have 
\begin{align}\label{part33}
  \sum_{m=1}^{n+1}  \sum_{k \in K^m_1} \frac{1}{\sin^2  \frac{v_{k}-\phi-w_{m} + \theta}{2}} \leq \frac{4(n+1)(n+3)^2}{\pi}.
\end{align}
Let $K^m_1:=\{1,2,\ldots, \frac{n}{2}+1\} \setminus V^m_1$. If $\theta \in [\frac{c}{n}, \frac{\pi}{2}]$, then by (\ref{part31}), (\ref{part32}) and (\ref{part33}),
\begin{align*}
  \sum_{m \in M}  \sum_{k\in K^m_1 \cap K^m_3} \frac{8\sin v_k  }{(n+2)(n+3)} A_1 \le {\mathcal{O}}(n^2).
\end{align*}
An analogous property can be shown for $A_2$, $A_3$ and $A_4$. Due to the similarity of the proofs, we omit the details. Thus we have 
\begin{align}\label{A}
  \sum_{m \in M}  \sum_{k\in K} \frac{8\sin v_k  }{(n+2)(n+3)} (A_1+A_2+A_3+A_4) \le {\mathcal{O}}(n^2).
\end{align}
 To consider all the components of $I_1$, it is necessary to examine the following terms
\begin{align*}
  &B_1:=\left| \ \frac{ \cos \frac{w_m}{2} \cos \frac{\theta}{2} }{ 2\sin \phi \sin  \theta} \left( \frac{\cos \frac{(2n+3) (v_{k}-\phi) - 2(w_{m} - \theta)}{2} }{8\sin \frac{v_{k}-\phi+w_{m} - \theta}{2}} -   \frac{\cos \frac{(2n+3) (v_{k}-\phi) - 2(w_{m} + \theta)}{2} }{8\sin \frac{v_{k}-\phi+w_{m} + \theta}{2}} \right) \right|, 
  \\ 
  &B_2:=\left| \ \frac{ \cos \frac{w_m}{2} \cos \frac{\theta}{2} }{ 2\sin \phi \sin  \theta} \left( \frac{\cos \frac{(2n+3) (v_{k}+\phi) + 2(w_{m} - \theta)}{2} }{8\sin \frac{v_{k}+\phi-w_{m} + \theta}{2}} -   \frac{\cos \frac{(2n+3) (v_{k}+\phi) + 2(w_{m} + \theta)}{2} }{8\sin \frac{v_{k}-\phi-w_{m} - \theta}{2}} \right) \right|, 
  \\
  &B_3:=\left| \ \frac{ \cos \frac{w_m}{2} \cos \frac{\theta}{2} }{ 2\sin \phi \sin  \theta} \left( \frac{\cos \frac{(2n+3) (v_{k}-\phi) + 2(w_{m} + \theta)}{2} }{8\sin \frac{v_{k}-\phi-w_{m} - \theta}{2}} -   \frac{\cos \frac{(2n+3) (v_{k}-\phi) + 2(w_{m} - \theta)}{2} }{8\sin \frac{v_{k}-\phi-w_{m} + \theta}{2}} \right) \right|, 
  \\
  &B_4:=\left| \ \frac{ \cos \frac{w_m}{2} \cos \frac{\theta}{2} }{ 2\sin \phi \sin  \theta} \left( \frac{\cos \frac{(2n+3) (v_{k}+\phi) - 2(w_{m} + \theta)}{2} }{8\sin \frac{v_{k}+\phi+w_{m} + \theta}{2}} -   \frac{\cos \frac{(2n+3) (v_{k}+\phi) + 2(w_{m} - \theta)}{2} }{8\sin \frac{v_{k}+\phi+w_{m} - \theta}{2}} \right) \right|. 
\end{align*}
As for $A_1$ one can show that
\begin{align*}
  B_1 &\leq \left|     \frac{ \cos \frac{w_m}{2} \cos \frac{\theta}{2} \sin \frac{v_{k}- \phi + w_{m}}{2} \cos \frac{\theta}{2}  \sin \frac{(2n+3)(v_{k}-\phi) -2w_m}{2}  \sin \theta}{ 8\sin \phi \sin  \theta \sin \frac{v_{k}-\phi+w_{m} - \theta}{2} \sin \frac{v_{k}-\phi+w_{m} + \theta}{2}}   \right|
   \\ &\quad \quad \quad + \left|     \frac{ \cos \frac{w_m}{2} \cos \frac{\theta}{2} \cos \frac{v_{k}-\phi + w_{m} }{2} \sin \frac{\theta}{2} \cos \frac{(2n+3)(v_{k}-\phi) -2w_m}{2}   \cos \theta  }{ 8\sin \phi \sin  \theta \sin \frac{v_{k}-\phi+w_{m} - \theta}{2} \sin \frac{v_{k}-\phi+w_{m} + \theta}{2}}   \right|.
\end{align*}
Hence
\begin{align}\label{part42}
  B_1 \leq   \frac{ 1 }{ 8 \sin  \phi  \sin^2 \frac{v_{k}-\phi+w_{m} - \theta}{2}} + \frac{1}{8  \sin  \phi  \sin^2 \frac{v_{k}-\phi+w_{m} + \theta}{2}}.
\end{align}
Let $K^m_2:=\{1,2,\ldots, \frac{n}{2}+1\} \setminus V^m_2$ and $K^m_6:=\{1,2,\ldots, \frac{n}{2}+1\} \setminus V^m_6$. Then
\begin{align}
  &\sum_{m=1}^{n+1}  \sum_{k \in K^m_2} \frac{1}{\sin^2  \frac{v_{k}-\phi + w_{m} - \theta}{2}} \leq \frac{4(n+1)(n+3)^2}{\pi}, \label{part43}\\
  &\sum_{m=1}^{n+1}  \sum_{k \in K^m_6} \frac{1}{\sin^2  \frac{v_{k}-\phi + w_{m} + \theta}{2}} \leq \frac{4(n+1)(n+3)^2}{\pi} \label{part44}
\end{align}
in the same manner as before. If $\phi \in [\frac{c}{n}, \pi - \frac{c}{n}]$, then by (\ref{part42}), (\ref{part43}) and (\ref{part44}),
\begin{align*}
  \sum_{m \in M}  \sum_{k\in K^m_2 \cap K^m_6} \frac{8\sin v_k  }{(n+2)(n+3)} B_1 \le {\mathcal{O}}(n^2).
\end{align*}
Similar inequalities can be proven for $B_2$, $B_3$ and $B_4$. Thus 
\begin{align}\label{B}
  \sum_{m \in M}  \sum_{k\in K} \frac{8\sin v_k  }{(n+2)(n+3)} (B_1+B_2+B_3+B_4) = {\mathcal{O}}(n^2).
\end{align}
Since $I_1 \leq \frac{8\sin v_k  }{(n+2)(n+3)} \sum\limits_{i=1}^{4} (A_i + B_i)$, (\ref{A}) and (\ref{B}) yields  
\begin{align}\label{I1}
  \sum_{m \in M}  \sum_{k\in K} I_1 \le {\mathcal{O}}(n^2).
\end{align}
Putting (\ref{I3}), (\ref{I2}) and (\ref{I1}) together, we have 
\begin{align}\label{pom222}
\sum_{m \in M} \sum_{k \in K}  K_n ((\theta,\phi);(w_m, v_k))  \le  \mathcal{O}(n^2).
\end{align} 
Since $M=\{1,2,\ldots,n+1\} \setminus \Theta$ and $K=\{k \in \mathbb{N}: 1 \leq k \leq \frac{n}{2}+1, \, k \notin \Omega\}$, $\lambda_n(\cos\theta,\cos\phi) \le {\mathcal{O}}(n^2)$, by (\ref{pom2}), (\ref{pom22}) and (\ref{pom222}).
It is worth emphasizing that if $\theta, \phi \in (\delta, \pi - \delta)$ for certain constant $\delta \in (0, \frac{\pi}{2})$ independent of $n$, then
\begin{align}\label{sI_1}
  \sum_{m \in M}  \sum_{k\in K} I_1 \le {\mathcal{O}}(n).
\end{align}
The remaining expressions under consideration have an analogous property. This means that 
\[ \lambda_n(\cos\theta,\cos\phi) \le {\mathcal{O}}(n) \]  
for $\theta, \phi \in (\delta, \pi - \delta)$.
\end{proof}

\begin{remark}
It is worth mentioning that Xu in \cite{Xu1, Xu2}
studied a generalization of Morrow-Patterson points for the weight functions
\begin{equation} \label{Xu_weights}
|x-y|^{2a+1}|x+y|^{2b+1}  (1-x^2)^{\pm {1 \over 2}}(1-y^2)^{\pm {1 \over 2}}
\end{equation}
for $a, b \ge -1/2$. In particular, for $a = b = -\frac{1}{2}$, they become Chebyshev weights of the first and the second kind. In the case of Chebyshev weight of the first kind, the author has discussed the corresponding cubature rules and interpolation formulas, including the Lebesgue constants growth for $n$ even in \cite{Xu1} and $n$ odd in \cite{Xu2}. In particular, it has been shown that the Lebesgue constants growth is in general of this type:
${\cal O}(n^{2 \max\{a,b\}+1})$ for $a,b> -1/2$ and
${\cal O}((\log n)^2)$ for $a=b <-1/2$. 

Concerning the second kind's Chebyshev weight, the Lebesgue constants growth has been recently studied in \cite{LZ} not for Morrow-Patterson points but for the so-called Xu points. The results show again a $n^2$ growth (see \cite[Th.1]{LZ}).
\end{remark}

\section{Numerical tests}
We developed a Matlab script, {\tt Leb\_MP.m}, to help readers to produce some numerical experiments. The script and two auxiliary functions (available on GitHub \url{https://github.com/demarchi-github/MP-points}) allow the construction and the plot of the Morrow-Patterson points for different values of the degree $n$, the Lissajous curve \eqref{MP_curve}, the Lebesgue function \eqref{leb_fun}, and the corresponding Lebesgue constant. The Lebesgue function can be evaluated on a grid of equally spaced points or Chebsyhev-Lobatto points. The latter is a weakly admissible mesh of the square and, as detailed in \cite{BKSV}, it is an optimal mesh for evaluating Lebesgue constants. 

In Figure \ref{last} we see the Lebesgue function for $n=30$ evaluated on a Chebsyshev-Lobatto grid of $10000$ points, and the growth of the Lebesgue constant $2\le n \le 30$, compared with the bound $(n+1)(n+2)/2$ and both curves are almost overlapping. Numerically the Lebesgue constant is attained at each corner of the square $Q$ and the computed values are $(n+1)(n+2)/2.$

In Figure \ref{last1} we show the plots of the Lebegsue function and the corresponding Lebesgue constant for the {\it Extended Morrow-Patterson} points, which are {\it scaled} Morrow-Patterson points \cite{marchi5}. Indeed, using our notations
$$x_m^{EMP}={x_m^{MP} \over \alpha_n}\,,\; y_k^{EMP}={y_k^{MP}\over \beta_n} $$ where $\alpha_n=\cos(\pi/(n+2)),\; \beta_n=\cos(\pi/(n+3))$. As expected the values are lower but still show a $n^2$ growth.

\subsection{Final remarks}
The numerical evidence shows that the Lebesgue function, with the notation of the paper $\lambda(x,y)$, is symmetric not only for $y$, but also to the origin. By Lemma \ref{lem1} above we proved that $\lambda(x,y)=\lambda(-x,y)$. However, it is an open problem of proving that $\lambda(x,y)=\lambda(-x,-y)$. Another open problem is to prove that the Lebesgue constant is attained at each corner of the square as suggested by the numerical results.

\begin{figure}[!h] 
\centering
\includegraphics[scale=0.35]{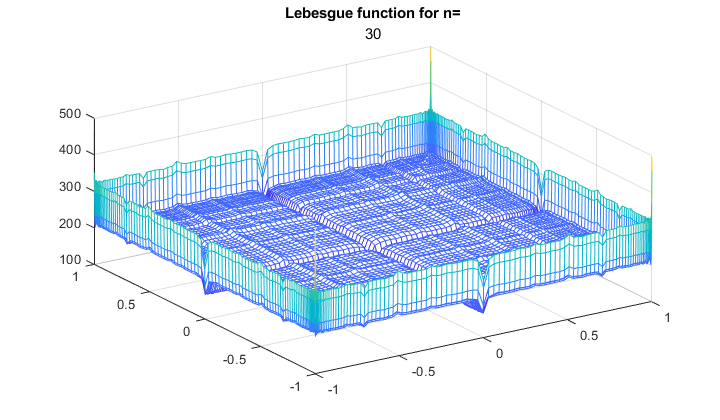} \hfill
\includegraphics[scale=0.35]{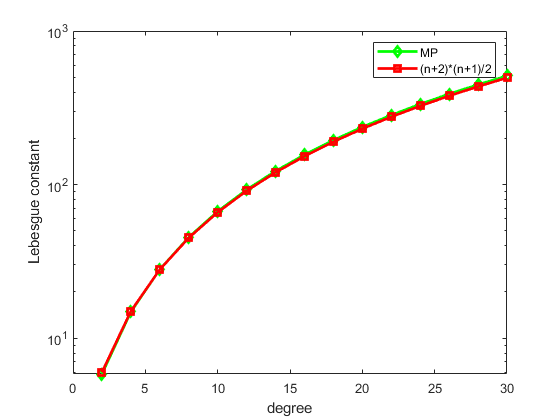} 
\caption{Top: the Lebesgue function for $n=30$. Bottom: the behavior of the Lebesgue constant and the bound $(n+1)(n+2)/2$. }
\label{last}
\end{figure}

\begin{figure}[!h] 
\centering
\includegraphics[scale=0.4]{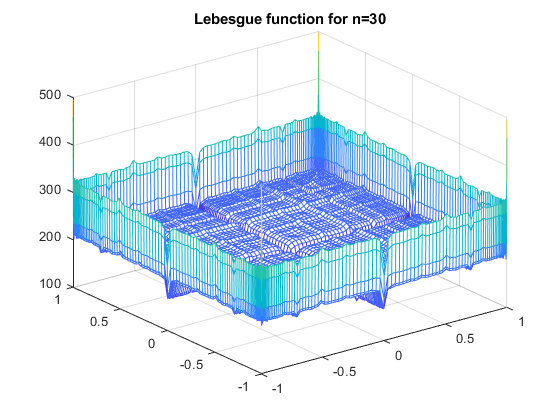} \hfill
\includegraphics[scale=0.35]{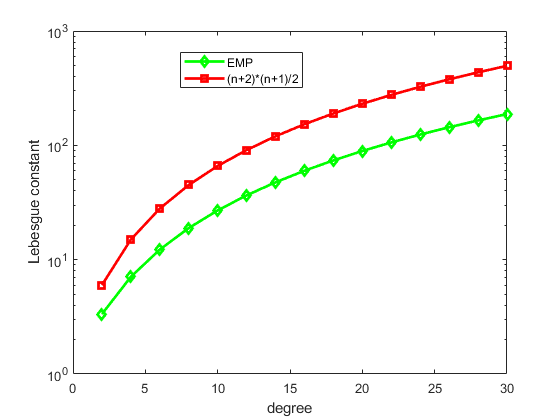} 
\caption{Extended Morrow-Patterson points. Top: the Lebesgue function for $n=30$. Bottom: the behavior of the Lebesgue constant compared with the bound $(n+1)(n+2)/2$. }
\label{last1}
\end{figure}

\vskip 3mm
{\bf Acknowledgments}
Work partially supported by the program Excellence Initiative – Research
University ID UJ at the Jagiellonian University in Kraków (Leokadia  Bialas-Ciez and Stefano De Marchi). The third author acknowledges the INdAM-GNCS group, RITA (Italian Network on Approximation), and the topic group on "Approximation Theory and Applications" of the Italian Mathematical Union.

\end{document}